\documentclass{amsart}

\DeclareSymbolFont{bbold}{U}{bbold}{m}{n}
\DeclareSymbolFontAlphabet{\mathbbold}{bbold}

\usepackage{amsmath}%

\usepackage{amsfonts}%
\usepackage{amssymb}%
\usepackage{graphicx}
\usepackage{bm}

\usepackage{mathabx}
\usepackage{rotating}
\usepackage[all,cmtip]{xy}
\usepackage{cite}
\usepackage{hyperref}
\usepackage{geometry}
\usepackage{bbm}
\usepackage{stmaryrd}
\usepackage{tensor}
\usepackage{wasysym}
\usepackage{leftidx}

\newtheorem{maintheorem}{Theorem}

\newtheorem{dummy}{dummy}[section]
\newtheorem{assumptions}[dummy]{Assumptions}              
\newtheorem{lemma}[dummy]{Lemma}
\newtheorem{theorem}[dummy]{Theorem}
\newtheorem{corollary}[dummy]{Corollary}
\newtheorem{proposition}[dummy]{Proposition}

\theoremstyle{definition}                                  
\newtheorem{definition}[dummy]{Definition}
\newtheorem{example}[dummy]{Example}
\newtheorem{remark}[dummy]{Remark}

\newtheorem*{example*}{Example}
\newtheorem*{remark*}{Remark}

\DeclareMathOperator{\Spec}{Spec}

\DeclareMathOperator{\bbHom}{Hom}
\newcommand{\bbHOM}{\mathcal{H}om}

\DeclareMathOperator{\End}{End}

\newcommand{\module}{\mathrm{mod}}

\DeclareMathOperator{\Sym}{Sym}
\DeclareMathOperator{\Lie}{Lie}

\DeclareMathOperator{\Ext}{Ext}
\DeclareMathOperator{\Rep}{Rep}
\newcommand{\Ad}{Ad}

\newcommand{\Perv}{\mathbf{Perv}}

\newcommand{\too}{\longrightarrow}

\newcommand{\ls}[2]{\leftidx{^{#1}}{#2}{}}
\newcommand{\lrsub}[3]{\leftidx{_{#1}}{{#2}}{_{#3}}}
\newcommand{\lrsubsuper}[5]{\leftidx{_{#1} ^{#2}}{{#3}}{_{#4} ^{#5}}}

\newcommand{\GIT}{{/\! /}}

\newcommand{\dw}{\dot{w}}

\DeclareMathOperator{\RES}{{\bf Res}}
\DeclareMathOperator{\IND}{{\bf Ind}}

\newcommand{\ind}{\mathbf{ind}}
\newcommand{\res}{\mathbf{res}}

\newcommand{\ad}{\mathrm{ad}}

\newcommand{\St}[2]{\lrsub{#1}{\underline{\mathfrak{st}}}{#2}}
\newcommand{\Stw}[3]{\leftidx{_{#1}}{\underline{\mathfrak{st}}}{^{#3} _{#2}}}

\DeclareMathOperator{\ST}{{\bf St}}

\DeclareMathOperator{\st}{{\bf st}}

\newcommand{\GS}[1]{\widetilde{\fg}_{#1}}
\newcommand{\Fl}[1]{\mathcal F\ell_{#1}}

\newcommand{\Four}{\mathbb{F}}
\newcommand{\reg}{\mathrm{reg}}
\newcommand{\nonreg}{\mathrm{nonreg}}
\newcommand{\Levi}{\mathcal Levi}

\newcommand{\comm}{\mathfrak{comm}}

\newcommand{\coker}{\mathrm{coker}}

\newcommand{\fc}{\mathfrak{c}}
\newcommand{\fg}{\mathfrak{g}}
\newcommand{\fh}{\mathfrak{h}}

\newcommand{\fp}{\mathfrak{p}}
\newcommand{\fu}{\mathfrak{u}}

\newcommand{\fv}{\mathfrak{v}}
\newcommand{\fE}{\mathfrak{E}}
\newcommand{\fF}{\mathfrak{F}}
\newcommand{\fG}{\mathfrak{G}}
\newcommand{\fM}{\mathfrak{M}}
\newcommand{\fN}{\mathfrak{N}}
\newcommand{\fK}{\mathfrak{K}}
\newcommand{\fL}{\mathfrak {L}}

\newcommand{\fD}{\mathfrak{D}}

\newcommand{\fz}{\mathfrak{z}}
\newcommand{\fl}{\mathfrak{l}}
\newcommand{\fq}{\mathfrak{q}}
\newcommand{\fn}{\mathfrak{n}}
\newcommand{\fm}{\mathfrak{m}}
\newcommand{\ffi}{\mathfrak{i}}
\newcommand{\fj}{\mathfrak{j}}
\newcommand{\fk}{\mathfrak{k}}
\newcommand{\fX}{\mathfrak{X}}
\newcommand{\fgl}{\mathfrak{gl}}

\newcommand{\cC}{\mathcal C}
\newcommand{\cD}{\mathcal D}
\newcommand{\cE}{\mathcal E}
\newcommand{\cF}{\mathcal F}

\newcommand{\cH}{\mathcal H}

\newcommand{\cK}{\mathcal K}
\newcommand{\cL}{\mathcal L}

\newcommand{\cN}{\mathcal N}
\newcommand{\cO}{\mathcal O}

\newcommand{\cQ}{\mathcal Q}

\newcommand{\cS}{\mathcal S}

\newcommand{\cU}{\mathcal U}

\newcommand{\bD}{\mathbf{D}}
\newcommand{\bM}{\mathbf{M}}

\newcommand{\A}{\mathbb A}

\newcommand{\C}{\mathbb C}
\newcommand{\D}{\mathbb D}

\newcommand{\F}{\mathbb F}
\newcommand{\G}{\mathbb G}

\newcommand{\Z}{\mathbb Z}

\newcommand{\uG}{\underline{G}}
\newcommand{\ug}{{\underline{\mathfrak{g}}}}
\newcommand{\ul}{{\underline{\mathfrak{l}}}}
\newcommand{\up}{\underline{\mathfrak{p}}}

\newcommand{\uq}{\underline{\mathfrak{q}}}
\newcommand{\um}{{\underline{\mathfrak{m}}}}
\newcommand{\uk}{{\underline{\mathfrak{k}}}}

\newcommand{\uh}{\underline{\mathfrak{h}}}

\newcommand{\ucN}{\underline{\mathcal{N}}}

\makeatletter
\newcommand*\leftdash{\rotatebox[origin=c]{-45}{$\dabar@\dabar@\dabar@$}}
\newcommand*\rightdash{\rotatebox[origin=c]{45}{$\dabar@\dabar@\dabar@$}}
\makeatother

\newcommand{\quot}[3]{{#1}\backslash{#2}/{#3}} 

\newcommand{\adjquot}{{/_{\hspace{-0.2em}ad}\hspace{0.1em}}}

\title{Generalized Springer Theory for $D$-modules \\on a Reductive Lie Algebra}
\author{Sam Gunningham}

\calclayout
\begin{document}
	\begin{abstract}
		Given a reductive group $G$, we give a description of the abelian category of $G$-equivariant $D$-modules on $\fg=\Lie(G)$, which specializes to Lusztig's generalized Springer correspondence upon restriction to the nilpotent cone. More precisely, the category has an orthogonal decomposition in to blocks indexed by cuspidal data $(L,\cE)$, consisting of a Levi subgroup $L$, and a cuspidal local system $\cE$ on a nilpotent $L$-orbit. Each block is equivalent to the category of $D$-modules on the center $\fz(\fl)$ of $\fl$ which are equivariant for the action of the relative Weyl group $N_G(L)/L$. The proof involves developing a theory of parabolic induction and restriction functors, and studying the corresponding monads acting on categories of cuspidal objects. It is hoped that the same techniques will be fruitful in understanding similar questions in the group, elliptic, mirabolic, quantum, and modular settings. 
	\end{abstract}

	\maketitle
	

	\subsection*{Main results}
	In his seminal paper \cite{lusztig_intersection_1984}, Lusztig proved the Generalized Springer Correspondence, which gives a description of the category of $G$-equivariant perverse sheaves on the nilpotent cone $\cN_G \subseteq \fg = \Lie(G)$, for a reductive group $G$:
	\[
	\Perv_G(\cN_G) \simeq \bigoplus^\perp_{(L,\cE)} \Rep(W_{G,L}).
	\]
The sum is indexed by \emph{cuspidal data}: pairs $(L,\cE)$ of a Levi subgroup $L$ of $G$ and simple cuspidal local system on a nilpotent orbit for $L$, up to simultaneous conjugacy. For each such Levi $L$, $W_{G,L} = N_G(L)/L$ denotes the corresponding relative Weyl group. 
	
	The main result of this paper is that Lusztig's result extends to a description of the abelian category $\bM(\fg)^G$ of all $G$-equivariant $D$-modules on $\fg$:
	\begin{maintheorem}\label{theoremabelian}
		There is an equivalence of abelian categories:
		\[
		\bM(\fg)^G \simeq \bigoplus^\perp_{(L,\cE)} \bM\left(\fz(\fl)\right)^{W_{G,L}},
		\]
		where the sum is indexed by cuspidal data $(L,\cE)$.
	\end{maintheorem}
	Here $\fz(\fl)$ denotes the center of the Lie algebra $\fl$ of a Levi subgroup $L$ which carries an action of the finite group $W_{G,L}$,\footnote{In fact, in the cases when $L$ carries a cuspidal local system, $W_{G,L}$ is a Coxeter group and $\fz(\fl)$ its reflection representation.}  and $\bM(\fz(\fl))^{W_{G,L}}$ denotes the category of $W_{G,L}$-equivariant $D$-modules on $\fz(\fl)$, or equivalently, modules for the semidirect product $\fD_{\fz(\fl)} \rtimes W_{G,L}$. If we restrict to the subcategory of modules with support on the nilpotent cone (which can be identified with the category of equivariant perverse sheaves via the Riemann-Hilbert correspondence), each block restricts to the category of equivariant $D$-modules supported at $\{0\} \subseteq \fz(\fl)$, which is identified with the category of representations of $W_{G,L}$; this recovers the Generalized Springer Correspondence. 
	\begin{example*}
		In the case $G=GL_n$, it is known that there is a unique cuspidal datum up to conjugacy, corresponding to a maximal torus of $GL_n$. Thus we have:
		\[
		\bM(\fg\fl_n)^{GL_n} \simeq \bM(\A^n)^{S_n} \simeq \fD_{\A^n} \rtimes S_n-\module
		\]
		
	\end{example*}
	\begin{remark*}
	The result for $GL_n$ can also be seen by using the functor of quantum Hamiltonian reduction, which, by the Harish-Chandra homomorphism of Levasseur-Stafford \cite{levasseur_invariant_1995,levasseur_kernel_1996}, naturally takes values in $\left(\fD_{\A^n}\right)^{S_n}-\module \simeq \fD_{\A^n} \rtimes S_n-\module$. In general, quantum Hamiltonian reduction will only see the Springer block (corresponding to the unique cuspidal datum associated with the maximal torus).
	\end{remark*}
	
	The proof of Theorem \ref{theoremabelian} is based on the idea that the category $\bM(\fg)^G$ can be described in terms of parabolic induction from \emph{cuspidal} objects in the category $\bM(\fl)^L$ associated to a Levi subgroup $L$ of $G$, where cuspidal objects are precisely those which themselves cannot be obtained from parabolic induction from a smaller Levi subgroup. The blocks of $\bM(\fg)^G$ corresponding to cuspidal data $(L,\cE)$ (for fixed $L$ and all corresponding $\cE$), consist precisely of summands of parabolic induction from cuspidal objects in $\bM(\fl)^L$.
		A priori, parabolic induction and restriction are a pair of adjoint triangulated functors between the equivariant derived categories $\bD(\fg)^G$ and $\bD(\fl)^L$. In this paper, we prove:
	\begin{maintheorem}\label{maintheoremindres}
		Parabolic induction and restriction restrict to bi-adjoint exact functors between abelian categories:
		\[
		\xymatrixcolsep{3pc}\xymatrix{
			\ind^G_{L}: \bM(\fl)^L \ar@/_0.3pc/[r] &  \ar@/_0.3pc/[l] \bM(\fg)^G:  
			\res_{L}^G
		}
		\]
		In addition, these functors are independent of the choice of parabolic subgroup $P$ containing $L$, and satisfy a Mackey formula:
		\[
		  \lrsub{M}{\st}{L} :=\res^G_{M} \ind^G_{L} (\fM) \simeq \bigoplus\limits_{w \in \quot{W_L}{W_G}{W_M}} \ind _{M \cap \ls \dw L} ^{M} \res^{\ls \dw L} _{M \cap \ls \dw L}
			{\dw _\ast}(\fM)
		\] 
	\end{maintheorem}
	When the object $\fM$ is cuspidal, the Mackey formula reduces to:
	\[
	\lrsub{L}{\st}{L}(\fM) \simeq \bigoplus_{w\in W_{G,L}} w_\ast(\fM).
	\]
	The core technical result of this paper states that the above isomorphism defines an isomorphism of monads acting on the category of cuspidal objects of $\bM(\fl)^L$.

	There is also a geometric characterization of the blocks corresponding to a fixed Levi $L$. The singular support of a $G$-equivariant $\fD_\fg$-module gives rise to a closed subvariety of the variety $\comm(\fg)$ of commuting elements of $\fg$. We define a locally closed partition of $\comm(\fg)$, indexed by 				conjugacy classes of Levi subgroups $L$ of $G$:
	\begin{equation*}\label{equationpartition}
	\comm(\fg) = \bigsqcup\limits_{(L)} \comm(\fg)_{(L)}.
	\end{equation*}
	In particular, we have a closed subset $\comm(\fg)_{\geq(L)}$ (respectively, $\comm(\fg)_{>(L)}$) given by the union of $\comm(\fg)_{(M)}$ for Levi subgroups $M$ which contain (respectively, properly contain) a conjugate of $L$.

	\begin{maintheorem}\label{maintheorempartition}	
		Given a non-zero indecomposable object $\fM \in \bM(\fg)^G$, the following are equivalent:
		\begin{enumerate}
		\item $\fM$ is a quotient of $\ind^G_L(\fN)$ for some cuspidal object $\fN$ in $\bM(\fl)^L$.
		\item $\fM$ is a direct summand of $\ind^G_L(\fN)$ for some cuspidal object $\fN$ in $\bM(\fl)^L$.
		\item The singular support of $\fM$ is contained in $\comm(\fg)_{\geq(L)}$, but not in $\comm(\fg)_{>(L)}$.
		\end{enumerate}
	\end{maintheorem}

The results can be interpreted as follows: the partition of $\comm(\fg)$ gives rise to a decomposition of the category of $G$-equivariant $D$-modules (which, a priori, is only semiorthogonal) by considering the subcategories of objects with certain singular support with respect to the partition. Theorem \ref{maintheorempartition} implies that this  agrees with the recollement situation given by parabolic induction and restriction functors, which in this case happens to be an orthogonal decomposition.
\begin{remark*}
	The category of \emph{character sheaves} consists of objects in $\bM(\fg)^G$ whose singular support is contained in $\fg\times \cN_G$. For such objects, the singular support condition in Theorem \ref{maintheorempartition} reduces to a support condition with respect to the corresponding (Lusztig) partition of $\fg$.
\end{remark*}

\subsection*{Background}
The study of equivariant differential equations and perverse sheaves on reductive groups (or Lie algebras) has a rich history. Given a complex reductive group $G$, Harish-Chandra showed that any invariant eigendistribution on (a real form of) $G$ satisfies a certain system of differential equations \cite{harish-chandra_invariant_1964}. The Harish-Chandra systems were reinterpreted by Hotta and Kashiwara \cite{hotta_invariant_1984} using $D$-module theory, and explained the connection to Springer theory (the geometric construction of representations of the Weyl group), as developed by Springer \cite{springer_construction_1978}, Kazhdan-Lusztig \cite{kazhdan_topological_1980}, and
Borho--MacPherson \cite{borho_partial_1983}. In his seminal paper \cite{lusztig_intersection_1984} Lusztig defined and classified cuspidal local systems and proved the generalized Springer Correspondence, which classifies equivariant perverse sheaves on the unipotent cone of $G$. He then went on to develop the theory of character sheaves, leading to spectacular applications to the representation theory of finite groups of Lie type. These ideas were recast in the $\cD$-module setting by Ginzburg \cite{ginsburg_admissible_1989,ginzburg_induction_1993} and Mirkovic \cite{mirkovic_character_2004} leading to many simplifications. More recent work of Rider and Russell explore the derived direction of generalized Springer theory \cite{achar_green_2011,rider_formality_2013,rider_perverse_2016}.

One aspect of the work presented here which is fundamentally different from that of previous authors, is that our results concern all equivariant $D$-modules, not just character sheaves (in particular, not just holonomic $D$-modules). The theory of character sheaves (as reformulated by Ginzburg) is about equivariant $\cD$-modules on $G$ for which the Harish-Chandra center $Z = (\fD_G)^{G\times G}$ acts locally finitely; thus character sheaves are discrete over the space $\Spec(Z)$ of central characters - character sheaves with incompatible central characters don't ``talk'' to each other. This becomes even clearer in the Lie algebra setting, where character sheaves can just be defined as the Fourier transforms of orbital $D$-modules (i.e. those with support on finitely many orbits). Many of our results in this paper have well-known analogues in the theory of character (or orbital) sheaves. For example, there are a variety of different proofs of exactness of parabolic induction and restriction (Theorem \ref{maintheoremindres}) in the setting of character and orbital sheaves (see e.g. Lusztig \cite{lusztig_character_1985}, Ginzburg \cite{ginzburg_induction_1993}, Mirkovic \cite{mirkovic_character_2004}, Achar \cite{achar_green_2011}); however, each of these proofs makes essential use of the character sheaf restriction, so they do not lead to a proof of Theorem \ref{maintheoremindres} (as far as the author is aware, the idea behind our proof is essentially new).

The change in focus from character sheaves to all equivariant $D$-modules is conceptually important as it allows one to do harmonic analysis on the category $\bM(\fg)^G$. For example, one might want to try to formulate a Plancherel theorem: express an arbitrary equivariant $D$-module as a direct integral of character sheaves.  

The setting of $G$-equivariant $D$-modules on $\fg$ is, in a certain sense, the simplest example in a family of settings:
\begin{itemize}
	\item \textbf{The mirabolic setting} Let $\fX = \fgl_n \times \C^n$. The group $GL_n$ acts on $\fX$ by the adjoint action on the first factor, and the standard representation on the second factor. We consider the category of $GL_n$-equivariant $D$-modules on $\fX$; more generally, we can consider the category of $c$-monodromic $D$-modules for any $c\in \C$. There is a well known relationship between such $D$-modules and modules for the spherical subalgebra of the rational Cherednik algebra (with parameter $c$); these ideas have been the subject of considerable interest, for notably in the work of Ginzburg with Bellamy, Etingof, Finkelberg, and Gan \cite{bellamy_hamiltonian_2015, etingof_symplectic_2002, finkelberg_mirabolic_2010, gan_almost_2006}. Our results in this paper suggest a new approach to this topic via parabolic induction and restriction functors. Work of McGerty and Nevins \cite{mcgerty_morse_2014} describes a recollement situation associated to a certain stratification of the cotangent bundle $T^\ast \fX$; it seems natural to construct this recollement from such induction and restriction functors. 
	\item \textbf{The group setting.} Most of the results in this paper carry over to the category of $G$-equivariant $D$-modules on the group $G$. This topic will be addressed in a sequel to this paper. One reason that this setting is more interesting is that this category appears as the value assigned to a circle by a certain topological field theory (TFT) $Z_G$. This TFT is constructed and studied in work of the author with David Ben-Zvi and David Nadler \cite{Ben-Zvi_character_2017} (based on earlier work of Ben-Zvi and Nadler \cite{ben-zvi_character_2009}), where we show that the value of $Z_G$ on an oriented surface $\Sigma$ is the Borel-Moore homology of the Betti moduli stack of $G$-local systems on $\Sigma$. The cohomology (and Hodge theory) of closely related varieties are the subject of conjectures of Hausel, Letellier, and Rodriguez-Villegas \cite{hausel_mixed_2008, hausel_arithmetic_2011}.
		\item \textbf{The quantum setting} The ring of \emph{quantum differential operators}, $\cD^q_G$ is a certain deformation of the ring of functions on $G\times G$. It has been studied extensively by Semenov-Tyan-Shanskii, Alekseev, Backelin--Kremnizer, and Jordan  \cite{alekseev_integrability_1993, backelin_quantum_2006, semenov-tian-shansky_poisson_1994, jordan_quantum_2009}. In the beautiful papers \cite{ben-zvi_integrating_2015, ben-zvi_quantum_2016} of Ben-Zvi--Brochier--Jordan it was shown that $\cD^q_G$ (or more precisely its category of equivariant modules $\bM_q(\uG)$) is obtained as the factorization homology on a genus one curve with coefficients in the category $\Rep_q(G)$ of representations of the quantum group.
	\item \textbf{The elliptic setting.} Let $G_E$ be the moduli of degree $0$, semistable $G$-bundles on an elliptic curve $E$, together with a framing at the identity element of $E$; this is a smooth variety with an action of $G$, changing the framing. The $G$-equivariant geometry of $G_E$ is closely related to the geometry of $G$ and $\fg$ (which can be defined in the same way as $G_E$, but replacing $E$ with a nodal or cusp curve). The category of $G$-equivariant $D$-modules on $G_E$ is the home of \emph{elliptic character sheaves}. The starting point of the study of vector bundles on elliptic curves was the work of Atiyah \cite{atiyah_vector_1957}, with subsequent work by Friedman--Morgan--Witten \cite{friedman_principal_1998}, and Baranovsky--Ginzburg (later with Evens) \cite{baranovsky_conjugacy_1996} \cite{baranovsky_representations_2003}. More recent work in this direction has been done by Ben-Zvi--Nadler \cite{ben-zvi_elliptic_2015}, Li--Nadler \cite{li_uniformization_2015}, and Fratila \cite{fratila_stack_2016}. Elliptic character sheaves can be thought of as a model for character sheaves for the loop group $LG$. Thus, the study of generalized Springer theory in this setting is relevant both to local and global (over $E$) versions of geometric Langlands. 
\end{itemize}

 \subsection*{Outline of the paper}

 	\begin{itemize}
 		\item In Section \ref{sectionmackey}, we define derived functors of parabolic induction and restriction, and show that they satisfy a general form of the Mackey theorem. 
 		\item In Section \ref{sectionreg} we study certain partitions of $\fg$ and the variety of commuting elements $\comm(\fg)$ indexed by Levi subgroups $L$ of $\fg$. 
 		\item In Section \ref{sectionproperties} we show that the parabolic induction and restriction functors restrict to exact functors on the level of abelian categories. 
 		\item In Section \ref{sectionabelian}, we show that the monad associated to the induction/restriction adjunction acting on a block of cuspidal objects is described by the action of the relative Weyl group. The monad can be thought of as a relative version of the Borel-Moore homology of the Steinberg variety with its convolution operation.\footnote{In this paper it will be convenient to make use of the categorical language of monads as a convenient tool. However, in our setting, the categories on which the monads act can all be represented as modules for an algebra. In this case, the notion of a monad is very concrete: a (colimit preserving) monad acting on the category $R-\module$ for a ring $R$ is an algebra object in the monoidal category of $R$-bimodules. This is the same thing as an \emph{$R$-ring}, i.e. a ring $A$ together with a  (not necessarily central) ring map $R\to A$. A module for the monad $A$ in $R-\module$ is nothing more than an $A$-module.} As in other incarnations of Springer theory, we first consider the restriction of the monad to the regular locus and then extend. In particular, this gives a new proof of the classical results of Springer theory. This allows us to deduce Theorem \ref{theoremabelian}. The existence of a recollement situation may be already deduced from the theory of induction and restriction; the fact that the recollement is split follows from the semisimplicity of the group algebra of the Weyl group.
 		\item In Appendix \ref{appendixcat}, we give an overview of some of the category theory required in this paper. In particular, we explain how the Barr-Beck theorem allows one to upgrade certain adjunctions to recollement situations.
 		\item In Appendix \ref{appendixdmod}, we gather some general results on $D$-modules.
 	\end{itemize}

\subsection*{Notation}
The following overview of notational conventions may be helpful when reading this paper.
\begin{itemize}
	\item The following is a summary of the notation for $D$-modules (further details can be found in Appendix \ref{appendixdmod}). The abelian category of (all) $D$-modules on a smooth variety or stack $X$ is denoted $\bM(X)$; the (unbounded) derived category is denoted $\bD(X)$. Thus if $U$ is a smooth algebraic variety with an action of an affine algebraic group $K$, and $X=U/K$ is the quotient stack, we write $\bM(X)$ or $\bM(U)^K$ for the abelian category of $K$-equivariant $D$-modules on $U$. The equivariant derived category $\bD(X)=\bD(U)^K$ is a $t$-category with heart $\bM(X)$ (which is not equivalent to the derived category of $\bM(X)$ in general).
	If $X$ is not smooth, then $\bM(X)$ means the subcategory of $D$-modules on some smooth ambient variety or stack with support on $X$.
	\item Throughout the paper, $G$ always refers to a complex reductive group $P$ and $Q$ are parabolic subgroups with unipotent radicals $U$ and $V$ respectively, and $L = P/U$, $M=Q/V$ are the Levi factors. Frequently, we will consider splittings of $L$ and $M$ as subgroups of $G$, and we usually want them to contain a common maximal torus. Lie algebras are denoted by fraktur letter $\fg, \fp , \fq$, etc. as usual.
	\item We use the notation $\ug$ to refer to the quotient stack $\fg/G$. Thus $\bM(\ug)$ means the same thing as $\bM(\fg)^G$. Similarly, we have $\up$, $\ul$, etc. This, and further notation for related stacks (e.g. the Steinberg stack $\St{Q}{P}$) is introduced in Subsection \ref{subsectionsteinberg}.
	\item Functors of parabolic induction and restriction (introduced in Subsection \ref{ssparabolic}) will be denoted by $\IND^G_{P,L}$ and $\RES^G_{P,L}$ in the derived category context, and $\ind^G_{P,L}$ and $\res^G_{P,L}$ in the abelian category context (once we have deduced that they are $t$-exact). The Steinberg functors are the composites of parabolic induction with restriction, and are denoted by $\ST$ (in the derived setting) or $\st$ (in the abelian setting).
	\item Given an element $x\in \fg$, we write $H(x)=H_G(x)$ for the centralizer $C_G(x_s)$ of the semisimple part of $x$ with respect to the Jordan decomposition. This is a Levi subgroup of $G$ (as it is the centralizer of the 1-dimensional subtorus of $\fg$ containing $x$; see \cite{borel_groupes_1965}, Th\'eor\`em 4.15).
	\item We frequently work with the poset $\Levi_G$ of Levi subgroups of $G$ up to conjugacy, ordered by inclusion. Thus $(M)\leq (L)$ means that some conjugate of $M$ is contained in $L$.   
	\item We use a superscript on the left to denote the adjoint action or conjugation. Thus $\ls gx$ means $\Ad(g)(x)$ (for some $g\in G$, $x\in \fg$), and if $P$ is a subgroup of $G$, then $\ls gP$ means $gPg^{-1}$.
\end{itemize}

	\subsection*{Acknowledgments}
	Much of this paper was written while I was a postdoc at MSRI, and I would like to thank them for their hospitality. This work (which arose from my PhD thesis) has benefited greatly from numerous conversations with various people over the last few years. The following is a brief and incomplete list of such individuals whom I would especially like to thank, with apologies to those who are omitted. P. Achar, G. Bellamy, D. Ben-Zvi, D. Fratila, D. Gaitsgory,  D. Jordan, D. Juteau,  P. Li, C. Mautner, D. Nadler, L. Rider, T. Schedler. I am particularly grateful to an anonymous referee for their comments and suggestions on an earlier draft.
	
    \subsection*{Note} This version of the paper has been revised and differs in places from the published paper. I would like to thank V. Ginzburg for drawing my attention to an error in an earlier version.

	\section{Mackey Theory}\label{sectionmackey}
	In this section we will define the functors of parabolic induction and
	restriction between categories of $D$-modules (or constructible sheaves), and define the Mackey filtration in that setting.
	
	\subsection{The adjoint quotient and Steinberg stacks}\label{subsectionsteinberg}
	Throughout this subsection, we will fix a connected reductive algebraic group $G$, and parabolic subgroups $P$ and $Q$ of $G$. We will denote by $U$ and $V$ the unipotent radicals of $P$ and $Q$, and the Levi quotients will be denoted $L=P/U$, $M=Q/V$, respectively. The corresponding Lie algebras will be denoted by lower case fraktur letters as usual; thus the Lie algebras of $G$, $P$, $Q$, $U$, $V$, $L$, $M$ shall be denoted $\fg$, $\fp$, ,$\fq$, $\fu$, $\fv$, $\fl$, $\fm$ respectively. 
	
	Recall that an algebraic group acts on its Lie algebra by the adjoint action. For ease of reading we will denote the adjoint quotient stacks with an underline as follows: $\fg/G = \ug$, $\fp/P =\up$, $\fl/L =\ul$, etc. We have a diagram of stacks
	\begin{equation}\label{diagramgs}
	\xymatrix{
		\ug & \ar[l]_r \up \ar[r]^s & \ul.
	}
	\end{equation}
	Of course, there is an analogous diagram involving $\uq$ and $\um$. The fiber product $\uq \times_{\ug} \up$ will be denoted by $\St QP$ and referred to as the \emph{Steinberg stack}. It is equipped with projections 
	\[
	\xymatrix{
		\um & \ar[l]_\alpha \St QP \ar[r]^\beta & \ul.
	}
	\]
	Explicitly we may write the Steinberg stack as a quotient
	\[
	\left\{ (x,g) \in \fg \times G \mid x\in \fq \cap \ls{g^{-1}}{\fp} \right\}/(Q\times P),
	\] 
	where the $(q,p) \in Q\times P$ acts by sending $(x,g)$ to $(\ls qx, qgp^{-1})$. In this realization, the morphisms $\alpha$ and $\beta$ are given by $\alpha(x,g) = (x + \fv) \in \fm$ and $\beta(x,g) = (\ls gx + \fu) \in \fl$. The Steinberg stack is stratified by the (finitely many) orbits of $Q\times P$ on $G$ and all of the strata have the same dimension. For each orbit $w$ in $\quot QGP$, we denote by $\Stw QPw$ the corresponding strata in $\St QP$.
	\begin{lemma}
		Given any lift $\dot{w} \in G$ of $w$, we have an equivalence of stacks 
		\[
		\Stw QPw \simeq (\fq \cap \ls \dw\fp) \adjquot (Q\cap \ls \dw P).
		\]
	\end{lemma}
	\begin{remark}
		The stack $\St QP$ is the bundle of Lie algebras associated to the inertia stack of $\quot QGP$; the equidimensionality of the stratification may be seen as a consequence of the orbit stabilizer theorem.
	\end{remark}
	
	\begin{remark}
		The morphism $s:\up \to \ul$ is not representable, however it is \emph{safe}, in the sense of Appendix \ref{appendixdmodfun}. Note that $s$ factorizes as
		\[
		\up = \fp/P \to \fl/P \to \fl/L = \ul.
		\]
		The first morphism is representable, and the second morphism gives rise to a derived equivalence of $D$-modules (but the $t$-structure is shifted). The benefit of working with such non-representable morphisms is that the shifts that usually appear in the definitions of parabolic induction and restriction are naturally encoded in our definition.
	\end{remark}
	
	To better understand Diagram \ref{diagramgs} and the Steinberg stack, let us consider the following spaces. The \emph{$P$-flag variety}, $\Fl{\fp} \simeq G/P$, is defined to be the collection of conjugates of $\fp$ in $\fg$ with its natural structure of a projective algebraic variety. The flag variety carries a tautological bundle of Lie algebras,
	\[
	\GS{P} = \left\{ (x, \fp^\prime) \in \fg \times \Fl \fp \mid x\in \fp^\prime \right\} \too \Fl{P}.
	\]
	The bundle of Lie algebras $\GS{P}$ has a naturally defined ideal $\GS{U}$ (whose fibers are conjugates of $\fu$) and corresponding quotient $\GS{L}$. The group $G$ acts on $\GS{P}$ and $\GS{L}$ and we have natural identifications $\GS{P}/G = \fp/P = \up$ and $\GS{L}/G = \fl \adjquot P$. 
	
	Now we can reinterpret Diagram \ref{diagramgs} as follows:
	\[
	\xymatrix{
		\fg/G  \ar[d] & \ar[l]_\rho \GS{P}/G \ar[r]^{\sigma^\prime} & \GS{L}/G \ar[d]_{\sigma^{\prime\prime}} \\
		\ug \ar[u]^\wr& \ar[l]_r \up \ar[u]^\wr \ar[r]^s & \ul .
	}
	\]
	The morphism $\sigma^\prime$ is representable and smooth of relative dimension $\dim(\fu)$, whereas $\sigma^{\prime \prime}$ is smooth of relative dimension $-\dim(\fu)$. Thus the morphism $s$ is smooth of relative dimension $0$. Note that $\sigma^{\prime\prime}$ induces an equivalence on the category of $D$-modules, which preserves the $t$-structures up to a shift of $\dim(\fu)$. Thus the top row of the diagram induces the same functor on $D$-modules up to a shift.
	
	The Steinberg stack may be written in terms of the varieties $\GS{Q}$ and $\GS{P}$ as follows:
	\[
	\St QP = \left(\GS{Q}\times_\fg \GS{P} \right)/G = \left\{(x,\fq^\prime,\fp^\prime) \in \fg\times \Fl Q\times \Fl P \mid x\in \fp^\prime \cap \fq^\prime \right\}/G.
	\]
	
	\subsection{Parabolic induction and restriction}\label{ssparabolic}
	The functors of parabolic induction and restriction are defined by:
	\[
	\xymatrixcolsep{3pc}\xymatrix{
		r_\ast s^! = \IND^G_{P,L}: \bD(\ul) \ar@/^0.3pc/[r] &  \ar@/^0.3pc/[l] \bD(\ug):  
		\RES_{P,L}^G= s_\ast r^!
	}
	\]
	where the notation is as in Diagram \ref{diagramgs} above.
	The morphism $r$ is proper, and $s$ is smooth of relative
	dimension $0$, thus $\IND_{P,L}^G$ is left adjoint to $\RES^G_{P,L}$.
	
	\begin{example}
	Let us make the definitions a little bit more explicit in the case in the case $P=B$ (a Borel subgroup) and $H = B/N$ (the canonical Cartan). Parabolic induction of an object $\fM \in \bM(\fh)$ is just given by 
	\[
	\IND^G_{B,H}(\fM) \simeq (\widetilde{\fg} \to \fg)_\ast(\widetilde{\fg} \to \fh)^!(\fM)[-\dim (\fn)]
	\]
	On the other hand, given an object $\fN\in \bM(\ug)$ (i.e. a $G$-equivariant $\fD_\fg$-module), parabolic restriction (forgetting about equivariance for now) is given by a derived tensor product:
	\[
	\RES^G_{B,H}(\fN) \simeq \fN \bigotimes\limits_{\Sym(\fn)\otimes \Sym(\fn)} \C [-\dim (\fn)].
	\]
	Here, we identify $\fg^\ast$ with $\fg$ using an invariant form, so that $\cO_{\fg} = \Sym(\fg^\ast)$ becomes identified with $\Sym(\fg)$. Then $\Sym(\fn)\otimes\Sym(\fn)$ is a commutative subalgebra of $\fD_{\fg}$ (the first factor of $\Sym(\fn)$ is really functions on $\fn^-$, and the second is constant coefficient differential operators on $\fn$), and so the relative tensor product above is computed by a Koszul complex. Later, we will see that the functor $\RES^G_{P,L}$ is $t$-exact, which due to the shifts, means that the Koszul complex is non-zero only in the middle degree!
	\end{example}

	We define the functor 
	\[
	\ST = \lrsubsuper{M,Q}{G}{\ST}{P,L}{} := \RES^G_{Q, M} \IND^G_{P,L}: \bD(\ul) \to \bD(\um),
	\]
	(we will often drop the subscripts and/or superscripts when the context is clear).
	By base change, we have $\ST \simeq   \alpha_\ast \beta^!$, where:
	\[
	\xymatrix{
		&& \St QP \ar[ld] \ar[rd]  \ar@/_2pc/[lldd]_{\alpha} 
		\ar@/^2pc/[rrdd]^{\beta}&& \\
		& \uq  \ar[ld] \ar[rd]  & \Diamond & \up  \ar[ld] \ar[rd]&\\
		\um & & \ug & & \ul.
	}
	\]
	Recall that the stack $\St QP$ has a stratification indexed by double cosets $\quot QGP$, with strata $\Stw QPw$. Thus the functor $\ST$ has a filtration (in the sense of Definition \ref{definitionfiltration}) indexed by the poset $\quot QGP$ (we will refer to this filtration as the Mackey filtration). Let $\ST^w = \alpha_{w\ast} \beta_w^!$ denote the functor associated to $w\in \quot QGP$, where $\alpha_w$ (respectively $\beta_w$) denotes the restrictions of $\alpha$ (respectively $\beta$) to $\Stw QPw$.

	\subsection{The Mackey Formula}
	For each $g\in G$, consider the conjugate parabolic $\ls gP$ in $G$ and its Levi quotient $\ls gL$. The image of $Q\cap \ls gP$ in $M$ is a parabolic subgroup of $M$ which we will denote $M\cap \ls gP$. The corresponding Levi quotient will be denoted $M\cap \ls gL$. Similarly, the image of $Q\cap \ls gP$ in $\ls gL$ is a parabolic subgroup, denoted $Q\cap \ls gP$, whose Levi subgroup is canonically identified with $M\cap \ls gL$. Analogous notation will be adopted for the corresponding Lie algebras. The conjugation morphism $\ls{g}{(-)}: \ul \to \ls g\ul$ gives rise to an equivalence $g_\ast: \bD(\ul) \to \bD(\ls g \ul)$.
	
	\begin{remark}
		If we choose Levi splittings of $L$ in $P$ and $M$ in $Q$ such that $M\cap \ls gL$ contains a maximal torus of $G$, then the notation above may be interpreted literally.
	\end{remark}
	
	\begin{proposition}[Mackey Formula]\label{propositionmackey}
		For each lift $\dw \in G$ of $w \in \quot QGP$, there is a natural equivalence of functors
		\[
		\ST^w \simeq \IND _{M\cap \ls \dw P, M \cap \ls \dw L} ^{M} \RES^{\ls \dw L} _{Q\cap \ls \dw L, M \cap \ls \dw L}
		{\dw _\ast}: \bD(\ul) \to \bD(\um).
		\]
	\end{proposition}
	\begin{proof}
		Recall that $\Stw QPw \simeq (\fq \cap \ls w\fp) \adjquot (Q\cap \ls wP)$. Consider the commutative diagram:
		\begin{align}\label{diagramintersect}
		\xymatrix{
			&&& \frac{\fq \cap \ls w\fp}{Q \cap \ls wP} \ar[ld] \ar[rd]
			\ar@/_3pc/[lllddd]_{\alpha_w}  \ar@/^3pc/[rrrrddd]^{\beta_w}&&&&\\
			&& \ar[rd]\ar[ld] \frac{\fm \cap \ls w\fp}{Q \cap \ls wP} & \Diamond &  \ar[ld]
			\frac{\fq \cap \ls w\fl}{Q \cap \ls wP} \ar[rd]&&&\\
			&\frac{\fm\cap \ls w\fp}{M\cap \ls wP} \ar[ld]\ar[rd]& \Diamond &  \frac{\fm \cap \ls
				w\fl}{Q \cap \ls wP} \ar[ld] \ar[rd]&\Diamond & \frac{\fq\cap \ls w\fl}{Q\cap \ls
				wL}\ar[ld] \ar[rd]&&\\
			\frac{\fm}{M} &&\frac{\fm \cap \ls w\fl}{M\cap \ls wP} \ar[rd]&  &  \frac{\fm \cap \ls
				w\fl}{Q \cap \ls wL} \ar[ld]  & &\frac{\ls w\fl}{\ls wL} & \frac \fl L \ar[l]^\sim _w\\
			&   &   & \frac{\fm \cap \ls w\fl}{M\cap \ls wL} &   &   &   & . 
		}
		\end{align}
		Traversing the correspondence along the top of the diagram (from right to left)
		gives the functor $\ST^w$, whereas traversing the two correspondences along the
		very bottom of the diagram gives the composite $\IND _{M\cap \ls \dw P, M \cap \ls \dw L } ^{M} \circ
		\RES^{\ls wL} _{Q\cap \ls \dw L, M \cap \ls \dw L } \circ \dw _\ast$. The three squares marked
		$\Diamond$ are all cartesian. The lowest square is not cartesian, however it is
		still true that the two functors $\bD(\fm\cap \ls w\fl)^{Q\cap \ls wL} \to \bD(\fm
		\cap \ls w\fl)^{M \cap \ls wP}$ obtained by either traversing the top or bottom of
		this square are naturally isomorphic (the base-change morphism is an isomorphism). Indeed,
		the stacks only differ by trivial actions of unipotent groups, thus all the $D$-module functors in the square are equivalences, differing at worst by a cohomological shift. The fact that the shift is
		trivial follows just from the commutativity of the diagram. Thus the entire
		diagram behaves as if it were a base change diagram as required.
	\end{proof}
	
	\subsection{Compatibility under repeated induction or restriction}
	Suppose $P'$ is a parabolic subgroup of $G$ which contains $P$. Then $L' := P'/U'$ has a parabolic subgroup $L'\cap P$ and the corresponding Levi quotient is canonically isomorphic to $L$ (one can either choose a splitting of $L'$ in $P'$, or just define $L'\cap P$ to be the image of $P$ in $L'$). 
		\begin{proposition}\label{propositionindresfactor}
		The functors of parabolic induction and restriction factor as follows:
		\begin{align*}
		\RES^G_{P,L} \simeq \RES^{L'}_{P\cap L', L}\RES^G_{P',L'} \\
		\IND^G_{P,L} \simeq \IND^G_{P',L'}\IND^{L'}_{P\cap L', L}
		\end{align*}
	\end{proposition}
\begin{proof}
	The appropriate diagram is now:
	\[
	\xymatrix{
		&& \up \ar[ld] \ar[rd]  \ar@/_2pc/[lldd]_{s} 
	\ar@/^2pc/[rrdd]^{r}&& \\
	& \up  \ar[ld] \ar[rd]  &  & \underline{\fp\cap\fl'}  \ar[ld] \ar[rd]&\\
	\ug & & \ul' & & \ul.
}
	\]
	As in the proof of Proposition \ref{propositionmackey}, the middle square is cartesian modulo unipotent gerbes. Thus the result follows from base change.
\end{proof}

	\section{Geometry of the adjoint quotient and commuting variety}\label{sectionreg}
	In this section, we will study certain loci of $\fg$ and the variety $\comm(\fg)$ of commuting elements of $\fg$. In the case $G=GL_n$ these loci record linear algebraic data concerning coincidences amongst the eigenvalues of matrices and simultaneous eigenvalues of commuting matrices.
	
	\subsection{The Lusztig partition of $\fg$}\label{subsectionpartition}
For each $x\in \fg$, let $H(x) = H_G(x) = C_G(x_s)$. This is a Levi subgroup of $G$.

\begin{definition}
 For each Levi subgroup $L$ of $G$, we write:
	\[
	\fg_{(L)} = \left\{ x\in \fg \mid \text{$H_G(x)$ is conjugate to $L$} \right\}.
	\]
\end{definition}
We will write $\fg_{\heartsuit}$ for the subset $\fg_{(G)}$; this is referred to as the \emph{cuspidal locus}. It consists of elements $x$ whose semisimple part is central in $\fg$. In other words, $\fg_{\heartsuit} = \fz(\fg)+\cN_G$. At the other extreme, $\fg_{(H)} = \fg_{rs}$ is precisely the subset of regular semisimple elements of $\fg$.

	These subsets define a partition of $\fg$ in to $G$-invariant, locally closed subsets, indexed by conjugacy classes of Levi subgroups. More precisely, there is a continuous map
	\begin{align*}
	\fg & \to \Levi_G^{op} \\
	x   & \mapsto (H(x))   
	\end{align*}
	where $\Levi_G^{op}$ is the poset of conjugacy classes of Levi subgroups, and $(L) \leq (M)$ means that some conjugate of $L$ is contained in $M$. We will use the notation $\fg_{\nleq(L)}$, $\fg_{>(L)}$, etc. to denote the corresponding subsets of $\fg$.

	\begin{example}\label{examplegln}
		In the case $G=GL_n$, conjugacy classes of Levi subgroups are in one-to-one correspondence with partitions of $n$. Given a partition $p = (p_1, \ldots, p_k)$, we may take $L_p = GL_{p_1} \times \ldots \times GL_{p_k}$. Then $\fgl_{(L_p)}$ consists of matrices with $k$ distinct eigenvalues $\lambda_1, \ldots, \lambda_k$ such that the dimension of the generalized eigenspace corresponding to $\lambda_i$ is $p_i$.
	\end{example}

Note that the stratification of $\fg$ is $G$-invariant, and the question of which subset an element $x\in \fg$ belongs to depends only on the semisimple part of $x$ (for example in Example \ref{examplegln} for $GL_n$, we saw that the partition depends only on the eigenvalues). Thus we are really just defining a partition $\fc = \sqcup \fc_{(L)}$ of the space $\fc = \fg\GIT G$ of semisimple conjugacy classes, such that $\fg_{(L)}$ is the preimage of $\fc_{(L)}$ under the characteristic polynomial map $\chi: \fg \to \fc$. This partition can be described in explicit linear algebraic (or combinatorial) terms as follows.

Let $H$ be a maximal torus of $G$ with Weyl group $W$. There are finitely many Levi subgroups of $G$ which contain $H$: they correspond bijectively with certain root subsystems of the corresponding root system of $(G,H)$ (they are precisely those generated by a subset of simple roots for some choice of polarization of the root system). For a Levi subgroup $L$ containing $H$, consider the linear subspace $\fz(\fl) \subseteq \fh$. The maximal such subspaces (corresponding to the minimal Levi subgroups) are the root hyperplanes of $\fh$, and more general subspaces of the form $\fz(\fl)$ are intersections of root hyperplanes. The stratification of $\fc$ can be described in terms of the images of these subspaces under the quotient map $\pi:\fh \to \fc = \fh\GIT W$.  Given an element of $x\in \fg_{\geq (L)}$, we can conjugate so that $x_s\in \fz(\fl)$; thence $\chi(x) = \chi(x_s) = \pi(x_s)$. Thus we have a commuting diagram:
\[
\xymatrix{
\fh \ar[r] & \fc \ar@ {} [dr]|\Square & \fg \ar[l] \\
\fz(\fl) \ar@{^{(}->}[u] \ar[r] &\fc_{\geq(L)}  \ar@{^{(}->}[u]& \ar[l] \fg_{\geq(L)} \ar@{^{(}->}[u]
}
\]
(Note that the left hand square is not cartesian in general: $\pi^{-1}(\fc_{\geq(L)})$ consists of the union of $W$-translates of $\fz(\fl)$.) Similarly, $\fc_{>(L)}$ is the image under $\pi$ of the union of $\fz(\fl')$ for Levi subgroups $L'$ which contain $L$. Using these ideas, we observe the following:
\begin{proposition}
The closure of $\fg_{(L)}$ is $\fg_{\geq(L)}$.
\end{proposition}
\begin{proof}
As the map $\chi:\fg\to \fc$ is open, we have that the preimage $\chi^{-1}$ commutes with the operation of taking closure. Thus it is sufficient to check that $\fc_{\geq (L)}$ is the closure of $\fc_{(L)}$. Moreover, as the map $\pi:\fh \to \fc$ is closed, we have that the operation of taking the image under $\pi$ commutes with closure. Thus we are reduced to showing that $\fz(\fl)$ is the closure of the subvariety give by the complement of the union of proper linear subspaces $\fz(\fl')$. But the complement of a proper linear subspace in $\fz(\fl)$ is open and dense, and thus so is a finite intersection of such, which shows the required closure property in $\fh$. 
\end{proof}

\subsection{The regular locus in a Levi}
Closely related to the Lusztig partion is the notion of \emph{regularity}.

\begin{definition}
Let $L$ be a Levi subgroup of $G$. An element $x\in \fl$ is called \emph{regular} (or $G$-regular) if $H_G(x) \subseteq L$. The set of regular elements in $\fl$ is denoted $\fl^{G-\reg}$ or just $\fl^\reg$. It's complement will be denoted $\fl^\nonreg$.
\end{definition}

The intersection of the $G$-regular locus of $\fl$ with the cuspidal locus of $\fl$ is given by 
\[
\fl^\reg_\heartsuit := \fl^\reg \cap \fl_\heartsuit = \fz(\fl)^\reg + \cN_L
\]
and consists of elements $x\in \fl$ such that $H_G(x)=L$.
By construction, we have that $\fg_{(L)}$ is the $G$-saturation of $\fl^\reg_{\heartsuit}$, and $\fg_{\leq(L)}$ is the $G$-saturation of $\fl^\reg$. 

\begin{remark}
If $H$ is a maximal torus of $G$, then $\fh^\nonreg$ is the union of root hyperplanes in $\fh$. More generally, for a Levi subgroup $L$, $\fz(\fl)^\nonreg$ is the union of hyperplanes $\fz(\fm_1), \ldots, \fz(\fm_k)$, where $M_1,\ldots, M_k$ are the minimal Levi subgroups of $G$ which contain $L$. However, $\fl^\nonreg$ is not cut out by linear equations in $\fl$ in general.
\end{remark}
	
\begin{remark}
The Lusztig \emph{stratification} (see \cite{lusztig_cuspidal_1995}, Section 6) is the refinement of the Lusztig partition indexed by pairs $(L,\cO)$ of a Levi subgroups $L$ and a nilpotent orbit $\cO$ in $\fl$, up to $G$-conjugacy. Given an element $x\in \fg_{(L)}$, we may assume (after conjugating $x$) that $H_G(x)=L$, and thus $x_n$ is a nilpotent element of $\fl$; we say that $x\in \fg_{(L,\cO)}$ if $x_n\in \cO$. As $\fg_{(L)}$ is the $G$-saturation of $\fl_{\heartsuit}^\reg = \cN_L + \fz(\fl)^\reg$, we see that $\fg_{(L,\cO)}$ is the $G$-saturation of $\cO + \fz(\fl)^\reg$.
\end{remark}

\begin{example}\label{examplestratificationgln}
 Continuing Example \ref{examplegln} for $G=GL_n$, note that partitions $q=(q_1,\ldots , q_\ell)$ which refine the partition $p$ index nilpotent orbits $\cO_q$ of $L_p$; the corresponding Lusztig stratum $\fgl_{(L_p,\cO_q)}$ consists of matrices with $k$ distinct eigenvalues with the dimension of generalized eigenspaces according to the partition $p$, and the size of the Jordan blocks prescribed by the refined partition $q$.
\end{example}

\subsection{\'Etale maps and Galois covers}\label{sec:etale-maps-and-galois-covers}
Here we record some results about the restriction of the diagram $\ug \leftarrow \up \rightarrow \ul$ to the various loci defined above. These results (or rather their analogues in the group setting) may be found in Lusztig \cite{lusztig_intersection_1984}, but we include proofs for completeness.

Let us first observe the following preliminary result.
	\begin{lemma}\label{lemmaregular}
		Let $x\in \fl$. The following are equivalent:
		\begin{enumerate}
			\item $x \in \fl^\reg$
			\item $C_\fg (x_s) \subseteq \fl$
			\item $C_\fg (x) \subseteq \fl$
		\end{enumerate}
	\end{lemma}
	\begin{proof}
		Recall that $x\in \fl^\reg$ means that $C_G(x_s) \subseteq L$. It is standard that this is equivalent to (2): the centralizer $C_G(x_s)$ is a Levi subgroup, and $C_\fg(x_s)$ is its Lie algebra. 
		
		Now suppose $x\in \fg$ and consider $\alpha = \ad(x): \fg \to \fg$. Thus $C_\fg(x) = \ker(\alpha)$ and $C_\fg(x_s) = \ker(\alpha_s)$. Choose an invariant bilinear form, and write $\fg = \fl \oplus \fl^\perp$. Statement (2) is equivalent to $\ker(\alpha_s|_{\fl^\perp}) = 0$, whereas Statement (3) is equivalent to $\ker(\alpha|_{\fl^\perp}) = 0$. Thus the result follows from a standard fact in linear algebra, which we record in the following lemma.
	\end{proof}
	
	\begin{lemma}\label{lemmalinear}
		Suppose $\alpha \in \End(V)$ is a linear transformation of a finite dimensional vector space. Then $\alpha$ is invertible if and only if $\alpha_s$ is invertible.
	\end{lemma}
	
	Given a parabolic $P$ containing $L$ as a Levi factor, we define $\fp^\reg$ to be the preimage of $\fl^\reg$ under the projection $\fp \to \fl$. 
	\begin{proposition}
		The natural morphism
		\[
		\up^\reg \to \ul^\reg
		\]
		is an equivalence.
	\end{proposition}
	\begin{proof}
		The fibers of the morphism $\up \to \ul$ are the quotient $(x+ \fu)\adjquot U$. Thus, we need to show that that $U$ acts simply transitively on $x+\fu$, where $x\in \fl^\reg$ (here, we consider $\fl$ to be a subalgebra of $\fp$), or in other words, the following map is an isomorphism:
		\begin{align*}
		U & \to x + \fu     \\
		g & \mapsto \Ad(g)x 
		\end{align*}
		As $U$ is unipotent, it suffices to show that the the corresponding map
		\begin{align*}
		\fu & \to x + \fu        \\
		u   & \mapsto \ad(u)(x). 
		\end{align*}
		is an isomorphism If $\ad(u)(x) = \ad(v)(x)$, then $u-v$ is an element of $C_\fg(x) \cap \fu$, so must be zero. Thus the map is injective. Any injective morphism between affine spaces of the same dimension must be an isomorphism.
	\end{proof}

	\begin{proposition}\label{propositionetale}
		The morphism of stacks 
		\[
		d^\reg:\ul^\reg \simeq \up^\reg \to \ug
		\]
		is \'etale with image $\ug_{\leq(L)}$.
	\end{proposition}
	\begin{proof}
		\footnote{The idea behind this proof was communicated to me by Dragos Fratila.} A morphism of (derived) stacks is \'etale if and only if its relative tangent complex is acyclic. Given a point $x\in \fl^\reg$, the relative tangent complex at $x$ is given by the total complex of:
		\[
		\xymatrix{
			\fg & \ar[l] \fl \\
			\fg \ar[u]^{\ad(x)} & \ar[l] \fl \ar[u]_{\ad(x)}
		}
		\]
		(The columns are the tangent complex of $\ug$ and $\ul$ at the point $x$, and the horizontal maps is the map of complexes induced by $d:\ul \to \ug$.) We can write $\fg = \fp \oplus \fl^\perp$. The acyclicity of the complex is equivalent to the statement that $\ad(x)|_{\overline{\fu}}$ is invertible; this follows from the fact that $x$ is regular and Lemma \ref{lemmaregular}.
	\end{proof}
	\begin{remark}
		In more down-to-earth terminology, Proposition \ref{propositionetale} states that $\widetilde{\fg}_P \to \fg$ is \'etale over the regular locus.
	\end{remark}
	
	\begin{proposition}
		The map $d^\reg_\heartsuit: \ul^\reg_\heartsuit \simeq \up^\reg_\heartsuit \to \ug_{\geq (L)}$ is a $W_{G,L}$-Galois cover onto the open subset $\ug_{(L)}$.
	\end{proposition}
	\begin{proof}
		The inclusion of $\fl_\heartsuit^\reg$ in to $\fg_{(L)}$ induces a morphism of stacks: 
		\[
		\ul_\heartsuit^\reg/W_{G,L} \simeq \fl_\heartsuit^\reg/N_G(L) \to \ug_{(L)} \simeq \fg_{(L)}/G
		\]
		The proposition is equivalent to the statement that this is an isomorphism. To see this, note that every $x\in \fg_{(L)}$ is conjugate to an element of $\fl^\reg_\heartsuit$. Moreover, if $x\in \fl_\heartsuit^\reg$, $C_G(x) = C_{N_G(L)}(x) = C_{L}(x)$, from which the result follows.
	\end{proof}

 The results of this subsection can be summarized in the following diagram:\footnote{Given a parabolic $P$ with Levi factor $L$, we define $\fp_\heartsuit$ to be the preimage $\fl_\heartsuit + \fu$.}
	\begin{proposition}\label{propositiondiagram}
	There is a commutative diagram:
	
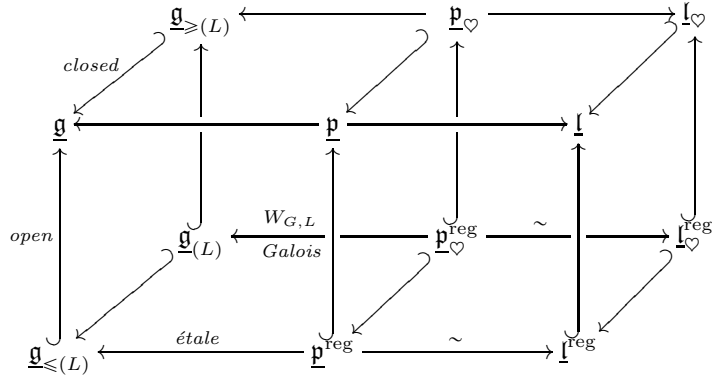
\begin{equation*}\label{diagramloci}
	\xymatrix{%
	%
		&\ug_{\geq(L)} \ar@{_{(}->}[dl]_-{\text{closed}}
	&&\ar@{->}[ll]\hole
	\up_\heartsuit
	\ar@{->}[rr]
	\ar@{_{(}->}[dl]
	&&\ul_\heartsuit\ar@{_{(}->}[dl]
	\\
		\ug&&\ar@{->}[ll]
		\up\ar@{->}[rr]
	&&\ul
	\\
	&\ug_{(L)}\ar@{_{(}->}[dl] 
\ar@{^{(}->}[uu]|!{[u];[u]}\hole 
&&\ar@{->}[ll]_(0.65){W_{G,L}}^(0.65){\text{Galois}}|!{[l];[l]}\hole
\up^\reg_{\heartsuit} \ar@{->}[rr]^(0.35){\sim}|!{[r];[r]}\hole
\ar@{_{(}->}[dl]\ar@{^{(}->}[uu]|!{[u];[u]}\hole
&&\ul^\reg_{\heartsuit}\ar@{_{(}->}[dl]\ar@{^{(}->}[uu]
	\\
	\ug_{\leq(L)} 
	\ar@{^{(}->}[uu]^{\text{open}}
&&\ar@{->}[ll]_{\text{\'etale}}
\up^\reg\ar@{^{(}->}[uu]
\ar@{->}[rr]^{\sim}
&&\ul^\reg \ar@{^{(}->}[uu]
	}
	\end{equation*}
	where:
	\begin{itemize}
	\item the vertical morphisms are open embeddings;
	\item the morphisms coming out of the page are closed embeddings; and
	\item all squares on the right hand side of the diagram, and all squares coming out of the page are cartesian.
	\end{itemize}
\end{proposition}

	\subsection{The Lusztig partition of the commuting variety}\label{subsectionlusztigcomm}
	In this subsection, we generalize the partition of Subsection \ref{subsectionpartition} to a partition of the variety $\comm(\fg)$ of commuting elements in $\fg$. Recall that given an element $x\in \fg$, $H(x) = H_G(x) =C_G(x_s)$ is defined to be the centralizer of the semisimple part of $x$. Similarly, given $(x,y) \in \comm(\fg)$, we define $H(x,y) = H(x) \cap H(y)$. It is the stabilizer of the pair of commuting semisimple elements $(x_s,y_s)$.
	
	Given a Levi subgroup $L$ of $G$, define
	\[
	\comm(\fg)_{(L)} = \{(x,y) \in \comm(\fg) \mid H(x,y) \quad \text{is conjugate to}\quad  L \},
	\]
	Note that $H(x)$ and $H(y)$ are Levi subgroups of $G$ which contain a common maximal torus (as $x_s$ and $y_s$ commute), thus their intersection is indeed a Levi subgroup of $G$. As for the partition of $\fg$, we write $\comm(\fg)_{\heartsuit}$ for $\comm(\fg)_{(G)}$, and (for example) $\comm(\fg)_{\nleq(L)}$ for the union of $\comm(\fg)_{(M)}$, where $(M) \nleq(L)$ in the poset $\Levi_G$.

\begin{example}\label{examplecommutinggln}
Continuing Examples \ref{examplegln} and \ref{examplestratificationgln}, let us consider the case $G=GL_n$. Given a partition $p = (p_1,\ldots,p_k)$ with corresponding Levi subgroup $L_p$, the subspace $\comm(\fgl_n)_{(L_p)}$ consists of pairs of commuting matrices $(x,y)$ with $k$ distinct simultaneous eigenvalues $\lambda_1, \ldots, \lambda_k$, such that the dimension of the $\lambda_i$-generalized simultaneous eigenspace (i.e. the simultaneous eigenspace for the commuting semisimple matrices $x_s$,$y_s$) is $p_i$.
\end{example} 
	
	\begin{remark}\label{remarkpartitioncommuting} Let us examine how the loci in the commuting variety interact with those in $\fg$. We have the following identities:
		\begin{itemize}
			\item $\comm(\fg)_{\heartsuit}= \left(\fg_{\heartsuit}\times \fg_{\heartsuit}\right) \cap \comm(\fg) = \fz(\fg)\times \fz(\fg) \times \comm(\cN_G)$. 
			\item $\comm(\fg)_{(L)} \cap \left(\fg\times \fg_{\heartsuit}\right) =\left( \fg_{(L)} \times \fg_{\heartsuit} \right)\cap \comm(\fg)$
		\end{itemize}
Moreover, there is an inclusion
\[
\comm(\fg)_{\nleq(L)} \subseteq \left(\fg_{\nleq(L)} \times \fg_{\nleq(L)}\right) \cap \comm(\fg)
\]
but it is not an equality in general. For example, take $G=GL_3$, and $L=GL_1^3$ to be the maximal torus, consisting of diagonal matrices. Then consider the pair $(x,y) \in \comm(\fgl_3)$, where $x=diag(0,0,1)$ and $y=diag(1,0,0)$. Then $H(x) = GL_2 \times GL_1$ and $H(y) = GL_1 \times GL_2$ (thought of as block matrices). Thus both $x$ and $y$ are contained inside $\fgl_{\nleq(GL_1^3)}$ (which is the complement of the regular semisimple locus). However $H(x)\cap H(y) = GL_1^3$, so $(x,y)$ is not contained in $\comm(\fgl_3)_{\nleq (GL_1^3)}$. In words, $x$ and $y$ are not regular semisimple matrices (they have repeated eigenvalues), but they are ``simultaneously regular semisimple'' (they have distinct simultaneous eigenvalues).
	\end{remark}
	
	Given a Levi subgroup $L$ of $G$, we define 
	\[
	\comm(\fl)^{\reg} = \{ (x,y) \in \comm(\fl) \mid H_G(x) \cap H_G(y) \subseteq L \}.
	\]
	For a parabolic subgroup $P$ containing $L$ as a Levi factor, we define $\comm(\fp)^{\reg}$ as the preimage of $\comm(\fl)^{\reg}$ under the canonical map $\comm(\fp) \to \comm(\fl)$. Similarly $\comm(\fl)^\nonreg$ is the complement of $\comm(\fl)^\reg$ in $\comm(\fl)$.
	
\begin{remark}
It is natural to ask whether the partition of $\comm(\fg)$ considered here refines to a stratification as in the case of $\fg$. In the case of $\fg$, this stratification arises essentially from the stratification of the nilpotent cone in to orbits (this is the theory of Jordan normal form in the $GL_n$ case). Following the analogy between $\fg$ and $\comm(\fg)$, we are reduced to finding a stratification of $\comm(\cN_G)$. However, unlike for $\cN_G$, there are infinitely many $G$-orbits on $\comm(\cN_G)$, so it is not so clear how to proceed. The variety $\comm(\cN_G)$ is studied by Premet \cite{premet_nilpotent_2003}, where it is shown that the irreducible components of $\comm(\cN_G)$ are in bijection with distinguished nilpotent orbits for $G$ (and moreover the variety is equidimensional).
\end{remark}

	\section{Properties of Induction and Restriction}\label{sectionproperties}
	In this section, we will prove that parabolic restriction and induction, restrict to exact functors on the level of abelian categories, and characterize the kernel of parabolic restriction in terms of singular support.
	
	\subsection{Induction and restriction over the regular locus}\label{subsectionregular}
		In this subsection we will keep the conventions of Section \ref{sectionmackey} (so $P$ and $Q$ are parabolic subgroups of $G$ with Levi quotients $L$, $M$ etc.), but additionally fix Levi splittings of $L$ and $M$ inside $P$ and $Q$.  Consider the set $\cS(M,L)$ which is defined to be the set of conjugates,$\ls gL$ of $L$ such that $M\cap \ls gL$ contains a maximum torus of $G$ (for convenience, let us assume that $M\cap L$ contains a maximum torus $H$ of $G$). There is a bijection $\quot{M}{\cS(M,L)}{L} \simeq \quot QGP$.
		
	According to the results of Subsection \ref{sec:etale-maps-and-galois-covers} (see Proposition \ref{propositiondiagram}), parabolic induction and restriction on the regular loci are just given by pullback and pushforward along the {\'e}tale map $d^\reg: \ul^\reg \to \ug$ (whose image is $\ug_{\leq(L)}$). More precisely, we have
		\[
		\RES^G_{P,L}(\fM)|_{\ul^\reg} \simeq \left(d^\reg\right)^!(\fM).
		\]

	This result has a number of important consequences. First let us give the following key definition:

	\begin{definition}
		An object $\fM \in \bD(\ug)$ is called \emph{cuspidal} if $\RES^G_{P,L}(\fM) \simeq 0$ for every proper parabolic subgroup $P$ of $G$. The full subcategory of cuspidal objects is denoted $\bD(\ug)_{cusp}$.
	\end{definition}

	\begin{proposition}\label{propositionsupport}
		Suppose $\fM \in \bD(\ug)$ and $\RES^G_{P,L}(\fM) \simeq 0$. Then $Supp(\fM) \subseteq \fg_{\nleq(L)}$. In particular, if $\fM$ is cuspidal, then $Supp(\fM)\subseteq \fg_{\heartsuit}$.
	\end{proposition}
	\begin{proof}
		Given $\fM$ with $\RES^G_{P,L}(\fM) \simeq 0$, we have
		\[
		d^{\reg,!}(\fM) \simeq \RES^G_{P,L}(\fM)|_{\ul^\reg} \simeq 0.
		\]
		But $d^\reg$ is \'etale with image $\ug_{\leq (L)}$, thus $\fM |_{\ug_{\leq (L)}} \simeq 0$. Thus $\fM$ is supported in $\ug_{\nleq(L)}$. Now, if $\fM$ is cuspidal, then $\RES^G_{P,L}(\fM)\simeq 0$ for every proper parabolic subgroup $P$. Thus $\fM$ is supported in the intersection of $\fg_{\nleq(L)}$ for all proper Levi subgroups $L$ of $G$; this intersection is precisely $\fg_{\heartsuit}$. Thus $\fM$ is supported in $\fg_{\heartsuit}$ as required.
	\end{proof}
	
	The next result explains how the Mackey filtration is naturally split on the regular locus. Recall that we have fixed another Levi subgroup $M$ inside a parabolic $Q$ (in addition to $L$ and $P$). 
	
	\begin{proposition}\label{propositionindresreg}
		Let $j: \um^\reg \hookrightarrow \um$ denote the inclusion. There is a canonical natural isomorphism:
		\[
		j^! \circ \lrsubsuper{M,Q}{}{\ST}{P,L}{} \simeq \bigoplus_{w\in \quot{M}{\cS(M,L)}{L}} j^! \circ \lrsubsuper{M,Q}{}{\ST}{P,L}{w}.
		\]
	\end{proposition}
	\begin{proof}
		We consider the open substack $\left(\St {Q}{P}\right)^\reg := \uq^\reg \times _{\ug} \up$ of $\St QP$. We denote by $\left(\Stw {Q}{P}{w}\right)^\reg$ the intersection of $\Stw QPw$ with $\left(\St {Q}{P}\right)^\reg$.
		Recall that the functor $\lrsubsuper{M,Q}{}{\ST}{P,L}{}$ is given by $\alpha_\ast \beta^!$ where:
		\[
		\xymatrix{
			\um & \St QP \ar[l]_\alpha \ar[r]^\beta & \ul
		}
		\]
		Thus $j^! \circ \lrsubsuper{M,Q}{}{\ST}{P,L}{w}$ is given by $\alpha^\reg _\ast \beta^!$ (where $\alpha^\reg: \St{Q}{P}^\reg \to \fm^\reg$). The result now follows from Lemma \ref{lemmastreg} below.
	\end{proof}
	
	\begin{lemma}\label{lemmastreg}
		Each stratum in the stratification
		\[
		\left(\St QP\right)^\reg \simeq \bigsqcup_{w\in \quot QGP} \left(\Stw QPw\right)^\reg.
		\]
		is both open and closed (i.e. the stratification is a disjoint union of connected components).
	\end{lemma}
	\begin{proof}
		Consider the opposite parabolic $\overline{Q}$ of $Q$ with respect to the Levi subgroup $M$ (with Lie algebra $\overline{\fq}$ etc.). We have isomorphisms $\uq^\reg \simeq \um^\reg \simeq \overline{\uq}^\reg$, and thus 
		\[
		\left(\St QP\right)^\reg \simeq \um^\reg \times_{\ug} \up \simeq \left(\St{\overline{Q}}{P}\right)^\reg
		\]
		Note that the bijections 
		\[
		\quot QGP \simeq \quot{M}{\cS(M,L)}{L} \simeq \quot {\overline{Q}}{G}{P}
		\]
		induce the opposite partial order on $\quot{M}{\cS(M,L)}{L}$. Thus the closure relations amongst the strata $\left(\Stw QPw\right)^\reg$ are self opposed. It follows that each stratum is both open and closed, as required.
	\end{proof}

\subsection{Fourier transform and parabolic restriction}
	The Fourier transform functor for $D$-modules defines an involution (see \cite{hotta_d_2008} 3.2.2, or \cite{ginzburg_lectures_1998} 4.8):
	\[
	\Four_\fg: \bD(\ug) \xrightarrow{\sim} \bD(\ug).
	\]
	\begin{lemma}[Lusztig, Mirkovi\'c \cite{mirkovic_character_2004}, 4.2]\label{lemmafourier}
		The functors of induction and restriction commute with the Fourier transform functor for all parabolic subgroups $P$ with Levi factor $L$:
		\begin{align*}
		\Four_\fl \RES^G_{P,L} \simeq \RES^{G}_{P, L} \Four_\fg   \\
		\Four_\fg \IND^{G}_{P, L} \simeq \IND^G_{P,L} \Four_\fl . 
		\end{align*}
	\end{lemma}
\begin{proof}[Sketch of proof]
	In the case of parabolic restriction, the lemma follows from the observation that the following diagram is cartesian, and taking duals swaps the right hand side with the left hand side:
	\[
\xymatrix{
&\fg \ar@{->>}[dr]&\\
\fp \ar@{^{(}->}[ur] \ar@{->>}[rd]&& \fg/\fu\\
&\fl \ar@{^{(}->}[ur]& 	
}\]
\end{proof}

\begin{lemma}\label{lemmasingsuppreszero}
Let $\fM$ be an object of $\bD(\ug)$. 
\begin{enumerate}
\item If $\RES^G_{P,L}(\fM)\simeq 0$, then the singular support of $\fM$ and of $\F(\fM)$ is contained in $\left(\fg_{\nleq(L)} \times \fg_{\nleq(L)}\right) \cap \comm(\fg)$. 
\item If $\fM$ is cuspidal, then the singular support of $\fM$ and of $\F(\fM)$ is contained in $\comm(\fg)_{\heartsuit}$.
\end{enumerate}
\end{lemma}
\begin{proof}
By Proposition \ref{propositionsupport} we have that $Supp(\fM) \subseteq \fg_{\nleq(L)}$. By Lemma \ref{lemmafourier}, we also have that $\RES^G_{P,L}(\F \fM) \simeq 0$, and thus $Supp(\F \fM) \subseteq \fg_{\nleq(L)}$. By Lemma \ref{lemmasingsuppFourier}, we see that $SS(\fM) \subseteq \fg_{\nleq(L)}\times \fg_{\nleq(L)}$; the same argument applies to $\F\fM$, which proves the first claim. If $\fM$ is cuspidal, then by the first part of the lemma, $SS(\fM)$ and $SS(\F\fM)$ are contained in $(\fg_{\heartsuit}\times \fg_{\heartsuit}) \cap \comm(\fg)$; but this is precisely $\comm(\fg)_{\heartsuit}$, as required (see Remark \ref{remarkpartitioncommuting}).
\end{proof}

	\subsection{Second adjunction}
	The following result was proved by Drinfeld-Gaitsgory \cite{drinfeld_theorem_2014} (using ideas developed by Braden \cite{braden_hyperbolic_2003}).
	\begin{theorem}\label{theoremsecond}
		Let $P^-$ denote the opposite parabolic subgroup of $P$ with respect to $L$. Then $\RES^G_{P^-,L}$ is left adjoint to $\IND^G_{P,L}$.
	\end{theorem}
	
	\begin{remark}\label{asdf}
		The theorem gives a ``cycle of adjoints'' of length 4:
		\[
		\ldots \dashv \IND^G_{P,L} \dashv \RES^G_{P,L} \dashv \IND^G_{\overline{P},L} \dashv \RES^G_{\overline{P},L} \dashv \IND^G_{P,L} \dashv \ldots.
		\]
	\end{remark}

	\begin{corollary}\label{corollarysecond}
	The functor $\RES^G_{P,L}$ preserves coherent $D$-modules.
	\end{corollary}

	\begin{proof}[Proof of Corollary \ref{corollarysecond}]
%
		
		Let us first note that $\bM(\ug)$ is a full subcategory of $\bM(\fg)$ (warning: this is not true for the equivariant derived category), which is the category of modules for the Noetherian ring $\fD_{\fg}$. In particular, coherent objects are the same as finitely generated modules, and every object in $\bM(\ug)$ is a union of its coherent subobjects.
		
		Given a coherent object $\fM\in \bM(\ug)$, consider the canonical map:
		\[
		\xymatrix{
			\fN := \res^G_{P,L} (\fM) \ar[r]^-\alpha _-\sim & \bigcup\limits_{i\in I} \fN_i,
		}
		\]
		where the union is over the coherent submodules $\fN_i$ of $\fN$. Consider the corresponding map under the (second) adjunction
		\[
		\fM \xrightarrow{\beta} \ind^G_{P^-,L} (\bigcup\limits_{i\in I} \fN_i) = \bigcup\limits_{i\in I} \ind^G_{P^-,L}(\fN_i).
		\]
		Here we have used the fact that $\ind^G_{P^-,L}$ commutes with filtered colimits (in this case, the union). As $\fM$ is coherent, and each $\ind^G_{P^-,L}\fN_i$ is coherent, the map $\beta$ factors through a \emph{finite} union 
		\[
		\xymatrix{
			\fM \ar[r]_-{\beta_0} \ar@/^2pc/[rr]^\beta & \bigcup\limits_{i\in I_0} \ind^G_{P^-,L}(\fN_i) \ar[r]& \bigcup\limits_{i\in I} \ind^G_{P^-,L} (\fN_i).
		}
		\]
		It follows that $\alpha$ must also factor through a finite union; in other words, $\fN$ is a finite union of coherent submodules so is itself coherent.
	\end{proof}
	
\subsection{Exactness of parabolic restriction}\label{ss:exact}
The goal of this section is to prove the following result.

\begin{theorem}[\cite{bezrukavnikov_parabolic_2018} Theorem 5.6]\label{thm:exact}
	The functors $\RES^G_{P,L}$ and $\IND^G_{P,L}$ are $t$-exact.
\end{theorem}

We will prove the left $t$-exactness of parabolic restriction in Proposition \ref{prop:left exact}. For the deduction of Theorem \ref{thm:exact} from Proposition \ref{prop:left exact} we follow \cite{bezrukavnikov_parabolic_2018}. Here is a brief outline of the argument. As $\RES^G_{P,L}$ is left $t$-exact, $\IND^G_{P,L}$ is right $t$-exact. But $\IND^G_{P,L}$ commutes with Verdier duality (it is the composite of a smooth pullback and proper pushforward). It follows that $\IND^G_{P,L}$ is also left $t$-exact. This is clear in the context of holonomic $D$-modules (where Verdier duality is a $t$-exact anti-auto-equivalence) but needs further justification in general - see \cite{bezrukavnikov_parabolic_2018}[Proof of Theorem 5.4]. By \cite{gunningham_generalized_2018}[Theorem 3.24] (Second Adjunction), we have that $\RES^G_{P^-,L}$ is right $t$-exact. But this argument holds for any parabolic $P$, and thus $\RES^G_{P,L}$ is right $t$-exact as required.

Thus we are reduced to proving the following result, and the remainder of the section is devoted to this proof.
\begin{proposition}\label{prop:left exact}
	If $\fG \in \bD(\ug)^{\geq 0}$, then for any parabolic $P$ with Levi factor $L$, $\RES^G_{P,L}(\fM) \in \bD(\ul)^{\geq 0}$.
\end{proposition}

Let $\fL = \RES^G_{P,L}(\fG) \in \bD(\ul)$ as in Proposition \ref{prop:left exact}. According to Lemma \ref{lem:stratifications}, to show that $\fL$ is concentrated in non-negative degrees it is enough to show that the $!$-restriction to $\ul_{(K)}$ is concentrated in non-negative degrees for each Levi subgroup $K$ of $L$ (with $\fk = \Lie(K)$). Fix such a subgroup $K$ and let
\[
\fL' := (\ul_{(K)} \to \ul)^!\fL
\]
denote the restriction. By \cite{gunningham_generalized_2018}[Proposition 2.13], we have an \'etale covering of stacks
\[
d^K_L:\uk_{\heartsuit}^{L-\reg} \to \ul_{(K)}
\]
where $\uk_\heartsuit^{L-\reg} \cong \fz(\fk)^{L-\reg} \times \ucN_K$. It follows that $\fL'$ is concentrated in non-negative degrees if and only if
\[
\fL'' := (\uk_{\heartsuit}^{L-\reg} \to \ul_{(K)})^! \fL'
\]
is concentrated in non-negative degrees. Note that 
\[
\uk_\heartsuit^{L-\reg} = \fz(\fk)^{L-\reg} \times \ucN_K
\] 
is partitioned in to locally closed subsets of the form 
\[
\fk_\heartsuit^M := \fz(\fm)^\reg \times \cN_K
\]
where $\fm = \Lie(M)$ and $M$ ranges over the (finite) set of Levi subgroups of $M$ containing $L$ such that $L\cap M = K$. Thus, applying Lemma \ref{lem:stratifications} again, we see that $\fL''$ is concentrated in non-negative degrees if and only if for each such $M$ 
\[
\fL''' = (\uk_\heartsuit^M \to \uk_\heartsuit^{L-\reg})^!\fL''
\]
is concentrated in non-negative degrees.

In summary we have shown the following:
\begin{lemma}\label{lem:strata by strata}
	Let $\fL \in \bD(\ul)$ be a bounded below complex. Then $\fL$ is concentrated in non-negative degrees if and only if
	$
	(\uk_\heartsuit^M \to \ul)^!(\fL)
	$
	is concentrated in non-negative degrees for each Levi subgroup $M$ of $G$ and $K$ of $L$ such that $K=M\cap L$.
\end{lemma}

Now fix $K,M$ Levi subgroups of $G$ with $L \cap M = K$ as above. Let $Q = M\cap P$; this is a parabolic subgroup of $M$ with Levi factor $K$. We write $\uq_\heartsuit^M$ for the substack $\fz(\fm)^\reg \times \ucN_Q$ of $\uq$. Consider the diagram of stacks:
\begin{equation}\label{diagram GPLHQK}
\xymatrix{
	\ug & \ar[l] \up \ar[r] & \ul\\
	\um_\heartsuit^\reg \ar[u] & \ar[l] \uq_\heartsuit^M \ar[r] \ar[u] & \uk_\heartsuit^M \ar[u]\\
}
\end{equation}

\begin{lemma}\label{lem:cartesian}
	The right hand square in Diagram \ref{diagram GPLHQK} is cartesian.
\end{lemma}
\begin{proof}
	Concretely, the lemma states that the following $P$-equivariant morphism of varieties is an isomorphism.\footnote{In this note, the associated bundle construction is always denoted by a superscript over the $\times$ symbol, and fiber products are written with a subscript.}
	\[
	\xymatrixrowsep{1pt}
	\xymatrix{
		P\times^Q \fq_\heartsuit^M \ar[r]^-\rho & \fp \times_{\fl} \left(L\times^K \fk_\heartsuit^M \right)\\
		(g,x) \ar@{|->}[r] & (\Ad(g)(x), \overline{g}, \overline{x})
	}
	\]
	Here, the overline notation denotes the projection from a parabolic subgroup or subalgebra to its Levi quotient. We will show that $\rho$ is bijective; we phrase the argument in terms of $\C$-points, but the same argument shows that $\rho$ is bijective on $R$-points for any commutative $\C$-algebra $R$ (alternatively, one could note that the source is connected and the target is normal, so bijection on $\C$-points implies isomorphism).
	
	Let us first note that the morphism $\rho$ is surjective. In other words, given an element $x\in \fp$ such that its image $y:=\overline{x} \in \fl$ lies in some $L$-conjugate of $ \fk_\heartsuit^M$, then we must show that $x$ lies in some $P$-conjugate of $\fq_\heartsuit^M$. But we can always conjugate $x$ by an element of $P$ so that $x_s$ lies in the Levi subalgebra $\fl$ and thus $x_s= y_s$. Conjugating further by an element of $L$, we may assume that $y_s \in \fz(\fm)^\reg$. Thus $x_n \in C_{\fp}(x_s) = \fp \cap \fm = \fq$, so $x = x_s + x_n \in \fq_\heartsuit^M$ as required. Similarly, to check that $\rho$ is injective boils down to the fact that if we have $(g,x) \in P \times \fq_\heartsuit^M$ such that $\ls gx = x$, then $g\in C_P(x) = P \cap C_G(x) \subseteq P \cap M = Q$. Thus $\rho$ is bijective at the set-theoretic level. As the target is normal and the source connected, it follows that $\rho$ is an isomorphism as required.
\end{proof}

We are now ready to finish the proof of Proposition \ref{prop:left exact}. Let $\fG \in \bD(\ug)^{\geq 0}$ and $\fL := \RES^G_{P,L}(\fG)$. We must show that $\fL$ is concentrated in non-negative degrees. By Lemma \ref{lem:strata by strata} it is enough to show that $(\uk_\heartsuit^M \to \ul)^!\fL$ is concentrated in non-negative degrees for each pair $M,K$ of Levis with $M\cap L = K$. Then by Lemma \ref{lem:cartesian} and base-change, we are reduced to showing that
\begin{equation}\label{to show}
\left(\RES^M_{Q,K}\right)^\reg_\heartsuit (\um_\heartsuit^\reg \to \ug)^! \fG
\end{equation}
is concentrated in non-negative degrees, where $\left(\RES^M_{Q,K}\right)^\reg_\heartsuit$ denotes the functor given by pull-push with respect to the bottom row of Diagram \ref{diagram GPLHQK}.
Note that $\um_\heartsuit^\reg \to \ug$ is a composition of an \'etale covering and a locally closed embedding, thus the corresponding $!$-pullback is left t-exact. Thus we are reduced to showing that $\left(\RES^M_{Q,K}\right)^\reg_\heartsuit$ is left t-exact. 

The bottom row of Diagram \ref{diagram GPLHQK} can also be written as
\[
\xymatrix{
	\fz(\fm)^\reg \times \ucN_M & \ar[l] \fz(\fm)^\reg \times \ucN_Q \ar[r] & \fz(\fm)^\reg \times \ucN_K
}
\]
Thus $\left(\RES^M_{Q,K}\right)^\reg_\heartsuit$ is an external product of the identity functor on $\bD(\fz(\fm)^\reg)$ and the usual parabolic restriction functor between $\bD(\ucN_M)$ and $\bD(\ucN_K)$. Thus we are reduced to showing that parabolic restriction is left $t$-exact for nilpotent orbital $D$-modules. 

At this point, we may appeal to one of the various proofs of (left) $t$-exactness for parabolic restriction of orbital (or character) sheaves in the literature, e.g. \cite{ginzburg_induction_1993}[Theorem 4.1]\cite{mirkovic_character_2004}[Theorem 7.3]. Alternatively, we may run the above argument again but now for the Fourier transform $\F_{\ug} \fG$, where $\fG$ is now assumed to be a nilpotent orbital $D$-module. In this way, we reduce to the situation where both $\fG$ and its Fourier transform are supported on the nilpotent cone. Such a $D$-module is readily seen to be cuspidal (see e.g. \cite{mirkovic_character_2004}[Theorem 4.7]) from which the result follows.

	\section{Generalized Springer Theory}\label{sectionabelian}
	In this section we apply the theory of parabolic induction and restriction to the give a description of the abelian category $\bM(\ug)$. 
	
	\subsection{The recollement situation associated to parabolic induction and restriction}\label{ssrecollement} To begin this section, we will show how our results on parabolic induction and restriction give rise to a recollement situation for the category $\bM(\ug)$. Later, we will see that the recollement is, in fact, an orthogonal decomposition.

\begin{definition}
	For a given parabolic subgroup $P$ of $G$ and Levi  factor $L \subseteq P$, we define $\bM(\ug)_{\nless(L)}$ to be the full subcategory of $\bM(\ug)$ consisting of objects $\fM$ for which $\RES^G_{P,L}(\fM)$ is cuspidal. We define $\bM(\ug)_{\nleq(L)}$ to be the further subcategory consisting of objects such that $\RES^G_{P,L}(\fM)$ is zero (recall that the zero object is considered cuspidal here).\footnote{We note that \cite{lusztig_intersection_1984}[Theorem 9.2] implies that if there exists a non-zero cuspidal object in $\bM(\ul)$, then all parabolic subgroups $P$ containing $L$ as a Levi factor are conjugate (see \cite{li_derived_2018}[Lemma 3.1]). It follows that these subcategories indeed depend only on the conjugacy class of $L$, independently of $P$.} 
\end{definition}

\begin{lemma}\label{lemmarestrictadjoint}
	Parabolic induction and restriction restrict to a pair of adjoint functors: 
	\[
	\xymatrixcolsep{5pc}
	\xymatrix{
		\ind^G_{P,L}: \bM(\ul)_{cusp} \ar@/_0.3pc/[r] &  \ar@/_0.3pc/[l] \bM(\ug)_{\nless(L)}:  
		\res_{P,L}^G
	} 
	\] 
\end{lemma}
\begin{proof}
	By definition, parabolic restriction takes objects of $\bM(\ug)_{\nless(L)}$ to cuspidal objects on $L$. On the other hand, if $\fM$ is a cuspidal object of $\bM(\ul)$, then it follows from the Mackey formula (Proposition \ref{propositionmackey}) that $\ind^G_{P,L}(\fM)$ is in $\bM(\ug)_{\geq(L)}$. Indeed, to check that $\res^G_{P,L}\ind^G_{P,L}(\fM)$ is cuspidal, take $M$ to be a Levi subgroup of $L$, and $Q$ a parabolic subgroup of $G$ contained in $P$ with $M$ as a Levi factor. Then the terms in the Mackey formula for $\res^G_{Q,M}\ind^G_{P,L}(\fM)$ all vanish, as required.
\end{proof}	

Note that the kernel of $\res^G_{P,L}$ restricted to $\bM(\ug)_{\nless(L)}$ is $\bM(\ug)_{\nleq(L)}$. Let us denote by $\bM(\ug)_{(L)}$ the quotient category of $\bM(\ug)_{\nless(L)}$ by $\bM(\ug)_{\nleq(L)}$. According to the results of Appendix \ref{appendixrecollement}, we may deduce: 

\begin{proposition}
There is a recollement situation:
\[
\xymatrixcolsep{7pc}
\xymatrix{
	\bM(\ug)_{\nleq(L)} \ar[r]^{\ffi_{\nleq(L)\ast}} & \ar@/_2pc/[l]_{\ffi_{\nleq (L)}^\ast}\ar@/^2pc/[l]^{\ffi_{\nleq(L)}^!} \bM(\ug)_{\nless(L)} \ar[r]^{\fj_{(L)}^\ast} & \ar@/^2pc/[l]^{\fj_{(L)!}} \ar@/_2pc/[l]_{\fj_{(L)\ast}} \bM(\ug)_{(L)},
}
\]
\end{proposition}
In particular:
\begin{itemize}
	\item Every object $\fM \in \bM(\ug)_{\nless (L)}$ fits in to an exact sequence:
	\begin{equation}\label{equationses}
	\xymatrix{ 
		\fj_{(L)!} \fj_{(L)}^\ast \fM \ar[r] &
		\fM \ar[r] &
		\ffi_{\nleq(L)\ast} \ffi_{\nleq(L)}^\ast \fM \ar[r] &
		0
	}
	\end{equation}
	\item The projection functor $\fj_{(L)!} \fj_{(L)}^\ast \fM$ is given by:
	\[
	\coker\left(\ind^G_{P,L}\res^G_{P,L}\ind^G_{P,L}\res^G_{P,L}(\fM) \too \ind^G_{P,L}\res^G_{P,L}(\fM) \right)
	\]
	where the map is given by the difference of the two morphisms coming from the counit of the adjunction. 
	\item 
The quotient category $\bM(\ug)_{(L)}$ can be identified via $\fj_{(L)!}$ with the subcategory of $\bM(\ug)$ consisting of quotients of objects which are parabolically induced from cuspidals on $\ul$. 
\item Any object $\fM \in \bM(\ug)_{\nless(L)}$ can be written as an extension of an object in $\bM(\ug)_{\nleq(L)}$ by a subquotient of an object in the essential image of $\ind^G_{P,L}|_{cusp}$.
\end{itemize}

It follows that we can express any object in $\bM(\ug)$ as an iterated extension of (subquotients of) objects which are parabolically induced from cuspidals on various Levis. This can be seen by the following algorithm:

\begin{enumerate}
	\item First, pick a minimal Levi $L$ such that $\fM \in \bM(\ug)_{\nless(L)}$. Note that $L$ could be equal to $G$ (if $\fM$ is cuspidal) or to the maximal torus $H$ (if $\res^G_{B,H}(\fM)$ is non-zero).
	\item If $L=G$, stop.
	\item Consider the exact sequence in (\ref{equationses}). 
	\item Replace $\fM$ by $\ffi_{\nleq(L) \ast} \ffi_{\nleq(L)}^\ast \fM$, and go back to step one.
\end{enumerate}
\begin{remark}
	In Subsection \ref{subsectionsteinbergweyl} we will see that the situation is much nicer: the monad $\lrsub{L,P}{\st}{P,L}$ restricted to cuspidal objects, is equivalent to the group monad of $W_{G,L}$. As explained in Appendix \ref{subsectionfinite}, this means that every object in the essential image of $\bM(\ug)_{(L)}$ is a direct summand (as opposed to just a quotient) of a parabolic induction from some cuspidal on $L$. 
\end{remark} 

\begin{remark}
	There is an analogous recollement situation in the derived setting too, see \cite{gunningham_derived_2017}.
\end{remark}

\subsection{The Springer block}\label{sec:the-springer-block}
In this subsection, we show how to deduce the ``non-generalized'' part of the generalized Springer theory presented in this paper (this is sufficient to prove the main results in the case $G=GL_n$). Fix a Borel subgroup $B$ containing a maximal torus $H$ with Weyl group $W$. To simplify notation we will write $\ind = \ind^G_{B,H}$, $\res=\res^G_{B,H}$, $\st = \res \circ \ind$, and $\St{}{}=\St{B}{B}$. 

Note that the adjunction $\ind,\res$ fits in to the set-up of Appendix \ref{appendixrecollement}. With that in mind, let $\bM(\ug)_{>(H)}$ denote the Serre subcategory of $\bM(\ug)$ consisting of objects killed by the functor $\res$, and $\bM(\ug)_{(H)}$ the quotient category of $\bM(\ug)$ by $\bM(\ug)_{>(H)}$. As explained in \ref{appendixrecollement}, we can identify $\bM(\ug)_{(H)}$ with the full subcategory of $\bM(\ug)$ consisting of quotients of objects in the essential image of $\ind$ (we refer to this subcategory as the \emph{Springer block}). Moreover, by the Barr-Beck Theorem we have the following:
\begin{lemma}\label{lemmabb}
The functor $\res$ identifies $\bM(\ug)_{(H)}$ with modules for the monad $\st = \res \circ  \ind$ acting on $\bM(\fh)$.
\end{lemma}

The Weyl group $W$ acts on $\bM(\fh)$, and $W$-equivariant objects are the same thing as modules for the monad $W_\ast$ (see Appendix \ref{subsectionfinite}). The following theorem captures the essence of the main results of this paper.
\begin{theorem}\label{theoremheartspringer}
	There is an equivalence of monads
	$
	\st \simeq W_\ast
	$.
\end{theorem}

Lemma \ref{lemmabb} then gives the following immediate corollary:
\begin{corollary}\label{corollaryabelianspringer}
	There is an equivalence of categories 
	\[
	\bM(\ug)_{(H)} \simeq \bM(\fh)^W \simeq \fD_{\fh}-\module
	\]
\end{corollary}
We recover the classical result of Springer theory from Theorem \ref{corollaryabelianspringer} by restricting to the subcategory $\bM(\ucN)$ of $\bM(\ug)$ consisting of objects supported on the nilpotent cone. 
\begin{corollary}[\cite{borho_partial_1983}]\label{corollaryspringernilp}
	There is an equivalence of categories 
	\[
	\bM(\ucN_G)_{Spr} \simeq \Rep(W)
	\]
In particular, the set of irreducible representations of $W$ inject in to the set of simple objects of $\bM(\ucN_G)_{Spr}$, which are indexed by pairs $(\cO,\cL)$ of a nilpotent orbit $\cO$ and a simple equivariant local system $\cL$ on $\cO$.
\end{corollary}
\begin{remark}
In the case $G=GL_n$ or $PGL_n$, the Springer block is the whole of $\bM(\ug)$, and thus the main result Theorem \ref{theoremabelian} is implied by Corollary \ref{corollaryabelianspringer} in this case. In general (e.g. for $G=SL_2$), there are other blocks that must be taken in to account.
\end{remark}

In order to prove Theorem \ref{theoremheartspringer}, we will identify the abelian category of colimit preserving endofunctors of $\bM(\fh) \simeq \fD_\fh-\module$ with $\fD_{\fh}$-bimodules, or equivalently, as objects in $\bM(\fh\times \fh)$. We refer to an objects of $\bM(\fh\times \fh)$ as \emph{integral kernels}. Given an integral kernel $\fK$ the corresponding functor is given by 
\[
\fM \mapsto p_{2\ast}(p_1^\circ(\fM) \otimes^\circ \fK)
\]
where $p_1,p_2:\fh\times \fh \to \fh$ are the projections and $\otimes^\circ$ denotes the tensor product of left $D$-modules given by $\otimes_{\cO}$ on the underlying $\cO$-modules. Similarly, the composition of endofunctors corresponds to \emph{convolution} of integral kernels:
\[
\fK_1 \ast \fK_2 := p_{13 \ast} \left(p_{12}^\circ(\fK_1) \otimes^\circ p_{23}^\circ(\fK_2)\right)
\]
where $p_{12}$, $p_{23}$, and $p_{13}$ are the projection maps $\fh^3 \to \fh^2$. Thus monads acting on $\bM(\fh)$ are the same as algebra objects in the monoidal category $\bM(\fh\times \fh)$, under convolution.\footnote{More algebraically, we could just think of a monad acting on $\bM(\fh)=\fD_{\fh}-\module$ as a $\fD_{\fh}$-ring. However, it will be more convenient to think of it geometrically as an object in $\bM(\fh\times \fh)$. For example, we will want to localize over certain open subsets of $\fh\times \fh$.}

Let $\fK_{\st}$ denote the integral kernel corresponding to the monad $\st$.
The integral kernel corresponding $w_\ast$ (for $w\in W$) is given by $\cO_w:= \Gamma_{w\ast} \cO_\fh$ where $\Gamma_w: \fh \to \fh \times \fh$ is the graph of the $w$-action. 
In this language, the Mackey formula says that $\fK_{\st} \in \bM(\fh\times \fh)$ carries a filtration such that the associated graded object is $\bigoplus_w \cO_w$. Theorem \ref{theoremheartspringer} states that the Mackey filtration on $\fK$ is canonically split, and moreover the splitting $\fK \simeq \bigoplus_w \cO_w$ is an isomorphism of monads.

\begin{remark}
	The corresponding statement is false in the derived category setting: the Mackey filtration is non-split even in the $SL_2$ case, as shown in \cite{gunningham_derived_2017}.
\end{remark}

\begin{remark}
One can check that the kernel $\fK_{\st}$ is given by $\cH^0(\alpha\times \beta)_\ast\omega_{\St{}{}}$, where
\[
\xymatrix{
	\uh 	&\St{}{}\ar[r]^-\alpha \ar[l]_-\beta &	\uh 
} 	
\]
are the projection maps on the Steinberg stack. Thus the $D$-module $\fK_{\st}$ over $\fh\times \fh$ records the fiberwise Borel-Moore homology (in a certain degree) of $\St{}{}$. The monad structure on the functor $\st$ reflects the convolution structure on Borel-Moore homology of the fibers. Theorem \ref{theoremheart} can be thought of as a ``relative'' version of the statement that the Borel-Moore homology of the nilpotent Steinberg stack (which is the fiber of $\St{}{}$ over $(0,0)\in \fh\times \fh$)  gives the group algebra of the Weyl group (see Chriss-Ginzburg \cite{chriss_representation_1997}). The proof uses the same basic principle (as do all proofs in Springer theory): the fact that there is an honest Weyl group action over the regular locus, and specialize to the non-regular locus. The difference between our approach and that of other authors, is that the focus is on the entire integral kernel representing the Steinberg functor, rather than the endomorphism algebra of a single object (the Springer sheaf).
\end{remark}

\subsection{Proof of Theorem \ref{theoremheartspringer}}\label{sec:proof-of-theorem-reftheoremheart}
The proof is divided in to three elementary lemmas: the first implies that there \emph{exists} a splitting of the Mackey filtration on $\fK_{\st}$; the second implies that there is a canonical splitting (as monads) over the regular locus; the third implies that such a splitting extends uniquely. 

\begin{lemma}\label{lemmaextvanish2}
	Let $v, w \in W$, $v \neq w$. Then 
	\[
	\Ext^i_{\bM(\fh\times \fh)}(\cO_v, \cO_w) = 0,
	\]
	for $i= 0,1$.
\end{lemma}

\begin{lemma}\label{lemmaregmonad}
	There is an isomorphism of algebra objects:
	\[
	\fK_{\st}|_{\fh^\reg \times \fh^{\reg}} \simeq \bigoplus_{w\in W} \cO_w |_{\fh^\reg\times\fh^\reg}.
	\]
\end{lemma}

\begin{lemma}\label{lemmaextend}
	We have:
	\[
	\bbHom_{\bM(\fh\times \fh)} (\cO_v, \cO_w) \simeq \bbHom_{\bM(\fh^\reg \times \fh^\reg)}(\cO_v \mid_{\reg}, \cO_w \mid_\reg)
	\]
\end{lemma}

Let us deduce Theorem \ref{theoremheartspringer} now, leaving the proofs of the various lemmas for the end of this subsection. Lemma \ref{lemmaextvanish2} shows that there must be \emph{some} splitting of the Mackey filtration, i.e. there exists an isomorphism $\fK_{\st} \simeq \bigoplus_w \cO_w$. Lemma \ref{lemmaregmonad} defines an isomorphism of monads
\[
\phi^\reg: \fK_{\st}|_{\fh^\reg \times \fh^\reg} \to \bigoplus_w \cO_w|_{\fh^\reg \times \fh^\reg}
\]
after restricting to the regular locus. Lemma \ref{lemmaextend} shows that this isomorphism extends to an isomorphism in $\bM(\fh\times \fh)$:
\[
\phi: \fK_{\st} \to \bigoplus_w \cO_w.
\]

It remains to show that the monad structures on $\st$ and on $W_\ast$ agree. The monad structure on the object $\bigoplus \cO_w$ is given by a collection of isomorphisms
\[
\tau_{v,w}: \cO_w \ast \cO_v \xrightarrow{\sim} \cO_{wv},
\]
for each pair $(v,w) \in W^2$. Similarly, the monad structure on $\st$ induces a (a priori different) monad structure on $\left(W_{G,L}\right)_\ast$ via the natural isomorphism $\phi$; we denote the corresponding structure morphisms by $\tau'_{v,w}$. Lemma \ref{lemmaregmonad} implies that $\tau = \tau'$ after restricting to the regular locus (we know $\phi$ is an isomorphism of monads there). On the other hand, by Lemma \ref{lemmaextend}, if $\tau$ and $\tau'$ agree on the regular locus, then they must agree on all of $\fh\times \fh$. Thus $\phi$ is an isomorphism of monads as required. This completes the proof of Theorem \ref{theoremheartspringer}.

\begin{proof}[Proof of Lemma \ref{lemmaextvanish2}]
	Let $\fh^{v,w} = \left\{ x \in \fh \mid v(x) = w(x) \right\}$. We have the cartesian diagram:
	\[
	\xymatrix{
		\fh^{v,w} \ar[r]^{\gamma_w} \ar[d]_{\gamma_v} & \fh \ar[d]^{\Gamma_v} \\
		\fh \ar[r]_{\Gamma_w} & \fh \times \fh
	}
	\]
	where $\gamma_w = \gamma_v$ is the inclusion of $\fh^{v,w}$ into $\fh$. By our assumption $w \neq v$, $\fh^{v,w} \neq \fh$. Let $d$ denote the codimension of $\fh^{v,w}$ in $\fh$.
	We compute:
	\begin{align*}
	R\bbHom_{\bD(\fh\times \fh)}( \cO_v,\cO_w) & =                                                                     
	R\bbHom_{\bD(\fh\times \fh)}(\Gamma_{v\ast} \cO_{\fh},\Gamma_{w\ast} \cO_{\fh}) \\
	& = R\bbHom_{\bD(\fh)}(\cO_{\fh}, \Gamma_v^!\Gamma_{w\ast}\cO_{\fh})    \\
	& = R\bbHom_{\bD(\fh)}(\cO_{\fh}, \gamma_{w\ast}\gamma^{!}_v \cO_{\fh}) \\
	& = R\bbHom_{\bD(\fh^{v,w})}(\gamma_w^\ast \cO_{\fh},                   
	\gamma_v^! \cO_{\fh})
	\\
	& = R\bbHom_{\bD(\fh^{v,w})}(\cO_{\fh^{v,w}}[d],                        
	\cO_{\fh^{v,w}} [-d])
	= H_{dR}^{\ast-2d}(\fh^{v,w}).
	\end{align*}
	As $d>0$, $\Ext^1( \cO_{v},\cO_{w})= H^{1-2d}_{dR}(\fh^{v,w}) = 0$ as required.
\end{proof}

\begin{proof}[Proof of Lemma \ref{lemmaregmonad}]
	By Proposition \ref{propositiondiagram}, the diagram which defines parabolic induction and restriction restricts to the $W_{G,L}$-Galois cover:
	\[
	\fh^\reg \times \ucN_L \simeq \ul_{\heartsuit}^\reg \simeq \ul_\heartsuit^\reg \to \ug_{(L)}
	\]
	In particular, the restriction of the Steinberg stack $\left(\St{P}{P}\right)_{(L)}$ is equivalent to the action groupoid of $W$ on $\ul^\reg_\heartsuit$. This directly translates in to the claimed result on $\fK_{\st}$.
\end{proof}	
The proof of Lemma \ref{lemmaextend} is a straightforward calculation (the result only depends on the fact that $\fh$ and $\fh^\reg$ are both connected).
\begin{remark}
	Lemma \ref{lemmaextend} does not hold in the derived category (i.e. the $R\bbHom$ complexes are not quasi-isomorphic). This is why one cannot deduce the corresponding statement to Theorem \ref{theoremheart} in the derived setting.
\end{remark}

	\subsection{Orthogonality of parabolic induction from non-conjugate Levi subgroups}
The following result is an important step in the orthogonal decomposition of Theorem \ref{theoremabelian}.
\begin{proposition}\label{propositionsplit}
	Suppose $L$ and $M$ are non-conjugate Levi subgroups. Given $\fN \in \bD^b_{coh}(\ul)_{cusp}$, and $\fM \in \bD^b_{coh}(\um)_{cusp}$ we have that $\IND^G_{P,L}(\fN)$ and $\IND^G_{Q,L}(\fM)$ are orthogonal, i.e.
	\[
	R\bbHom(\IND^G_{P,L}(\fN), \IND^G_{Q,L}(\fM)) \simeq R\bbHom(\IND^G_{Q,L}(\fM),\IND^G_{P,L}(\fN)) \simeq 0.
	\]
\end{proposition}
\begin{proof}
	First let assume $L$ is not conjugate to a subgroup of $M$. Let $\fN \in \bD^b_{coh}(\ul)_{cusp}$ and $\fM \in \bD^b_{coh}(\um)_{cusp}$. 
	As $M$ is not conjugate to a subgroup of $L$, for any $g\in G$, $M \cap \ls gL $
	is a proper subgroup of $\ls gL$. Thus 
	\[
	\lrsubsuper{Q,M}{}{\ST}{P,L}{w} \fN \simeq \IND _{M\cap \ls \dw P, M \cap \ls \dw L} ^{M} \RES^{\ls \dw L} _{Q\cap \ls \dw L, M \cap \ls \dw L}
	{\dw _\ast} \fN \simeq 0
	\]
	as $\dw_\ast\fN$ is cuspidal for each $\dw$. By the Proposition \ref{propositionmackey} (the Mackey formula), $\lrsubsuper{Q,M}{}{\ST}{P,L}{} \fN \simeq 0$. Thus 
	\begin{equation}\label{equationorth}
	R\bbHom (\IND_{Q,M}^G \fM, \IND^G_P \fN) = R\bbHom (\fM, \lrsubsuper{Q,M}{}{\ST}{P,L}{} \fN) \simeq 0.
	\end{equation}
	Note that the Verdier duality functor $\D_\ul$ preserves the category of cuspidal objects in $\bD(\ul)$ and the functor $\IND^G_{P,L}$ intertwines the Verdier duality functors. Applying \ref{equationorth} with $\fN$ and $\fM$ replaced by $\D_\ul \fN$ and $\D_\um \fM$, we obtain:
	\begin{align*}
	R\bbHom (\IND^G_{P,L} \fN, \IND^G_{Q,M} \fM) & = R\bbHom(\D_{\ug} \IND^G_{Q,M}\fM , \D_{\ug} \IND^G_{P,L} \fN)       \\
	& = R\bbHom(\IND^G_{Q,M} \D_{\um} \fM ,  \IND^G_{P,L} \D_{\ul}\fN) = 0, 
	\end{align*}
	Thus $\IND^G_{P,L}(\fN)$ is orthogonal to $\IND^G_{Q,M}(\fM)$ whenever $M$ is not conjugate to a subgroup of $L$. By switching the roles of $M$ and $L$, we obtain that $\IND^G_{P,L}(\fN)$ and $\IND^G_{Q,M}(\fM)$ are also orthogonal whenever $L$ is not conjugate to a subgroup of $M$. Thus $\IND^G_{P,L}(\fN)$ and $\IND^G_{Q,M}(\fM)$ are orthogonal whenever $M$ is not conjugate to $L$. (Note that $M$ is conjugate to $L$ if and only if $M$ is conjugate to a subgroup of $L$ and $L$ is conjugate to a subgroup of $M$).
\end{proof}

\begin{remark}\label{remarksummand}
	The proposition immediately implies that \emph{direct summands} of parabolic inductions from cuspidals on non-conjugate Levi subgroups are orthogonal. We will see in the following section that every object of the abelian category $\bM(\ug)_{(L)}$ is a direct summand of $\ind^G_{P,L}(\fM)$ for some cuspidal object $\fM \in \bM(\ul)$. This proves that the subcategories $\bM(\ug)_{(L)}$ are orthogonal for non-conjugate Levis, $L$. 
\end{remark}

In the proof above, we noted that the terms of the Mackey filtration of $\lrsubsuper{M,Q}{}{\ST}{P,L}{}(\fM)$ vanish unless $L$ is conjugate to a subgroup of $M$. Let us observe that if $M$ is conjugate to $L$, then the Mackey formula takes a special form:
\begin{proposition}\label{propositionmackeycuspidal}
	Suppose $\fM \in \bD(\ul)_{cusp}$. Then $\lrsubsuper{L,P}{}{\ST}{P,L}{w} \fM \simeq 0$ unless $w\in W_{G,L} = N_G(L)/L$ (which is naturally identified with a subset of $\quot{L}{\cS(L,L)}{L} \sim \quot PGP$). 
\end{proposition}
Thus $\lrsubsuper{L,P}{}{\ST}{P,L}{}(\fM)$ is an iterated extension of $w_\ast(\fM)$ where $w$ ranges over $W_{G,L}$ when $\fM$ is cuspidal.

	\subsection{The category of cuspidal objects}\label{sscuspidal}
	Let $\fg'= \fg/\fz(\fg) \simeq [\fg,\fg]$; thus $\fg \simeq \fz(\fg)\times \fg'$. Note that the torus $Z^\circ(G)$ acts trivially on $\fg$, and equivariance for a connected group acting trivially does not affect the abelian category of $D$-modules. Thus equivariance for $G$ is the same thing as equivariance for $G'=G/Z^\circ(G)$. In particular, given an object $\fM\in \bM(\fz(\fg)$ and $\fK\in \bM(\ug')$, we have an object $\fM \boxtimes \fK \in \bM(\ug)$.
	\begin{lemma}\label{lemmacuspidalnilpotent}
		An object $\fM\boxtimes \fK \in \bM(\ug)$ is cuspidal if and only if $\fK \in \bM(\ug')$ is cuspidal.
	\end{lemma}
	\begin{proof}
		Suppose $P$ is a parabolic subgroup of $G$ with Levi factor $L$. Set $P'=P/Z^\circ(G)$ and $L'=L/Z^\circ(G)$ (and similarly for the Lie algebras). We have the diagram:
		\[
		\xymatrix{
			\ug & \ar[l] \up \ar[r] & \ul \\
			\ar@{<->}[u]^\wr \fz(\fg) \times \ug' & \ar@{<->}[u]^\wr \ar[l] \fz(\fg) \times \up' \ar[r] & \ar@{<->}[u]^\wr \fz(\fg)\times \ul'
		}
		\]
		Thus, the functor of parabolic restriction factors as:
		\[
		\res^{G}_{P,L} \simeq \res^{G'}_{P',L'} \boxtimes id_{\bM(\fz(\fl))}.
		\]
		The lemma follows immediately from this observation.
	\end{proof}

\begin{definition}
The category of \emph{orbital sheaves} is defined to be $\bM(\ucN_G)$ (which we identify with a full subcategory of $\bM(\ug)$). The category of \emph{character sheaves} is defined to be the essential image of $\bM(\ucN_G)$ under Fourier transform. 
\end{definition}

Let $\fE= IC(\cO,\cE)$ be a simple orbital sheaf, where $\cE$ is a simple equivariant local system on a nilpotent orbit $\cO$. Such a local system $\cE$ is called cuspidal if $\fE$ is a cuspidal orbital sheaf in $\bM(\ug)$.\footnote{Note that the Fourier transform of $\fE$ is a cuspidal character sheaf, necessarily of the form $\cO_{\fz(\fl)}\boxtimes \fE'$ (with respect to the decomposition $\ug = \fz(\fg) \times \ug'$), where $\fE'$ is another simple cuspidal $D$-module on the nilpotent cone, and $\cO_{\fz(\fl)}$ is the trivial flat connection.} Given a cuspidal local system $\cE$, we define the full subcategory $\bM(\ug)_{(\cE)}$ of $\bM(\ug)$ to consist of objects supported on $\ucN_G\times \fz(\fg)$ of the form $\fE \boxtimes \fM$, where $\fM$ is any object in $\bM(\fz(\fg))$. 

\begin{proposition}\label{propositioncuspidal}
	The category of cuspidal objects decomposes as an orthogonal sum
	\[
	\bM(\ug)_{cusp} \simeq \bigoplus^\perp \bM(\ug)_{(\cE)},
	\]
	indexed by cuspidal data $(\cO,\cE)$. Moreover,
	$
	\bM(\ug)_{(\cE)} \simeq \bM(\fz(\fg))
	$
	via $\fM \boxtimes \fE \mapsto \fM$.
\end{proposition}
\begin{proof}
Recall that every cuspidal object of $\bM(\ug)$ is supported on $\ug_\heartsuit = \fz(\fg)\times \ucN_G$ (see Lemma \ref{lemmasingsuppreszero}). Also recall that the category $\bM(\ucN_G)$ is semisimple.\footnote{This follows from the parity vanishing condition in \cite{lusztig_character_1986} Theorem 24.8.
} 
Thus by Lemma \ref{lemmacuspidalnilpotent} any object of $\bM(\ug)_{cusp}$ can be written as a direct sum of objects of the form $\fE \boxtimes \fM$, where $\fE$ is a simple cuspidal object of $\bM(\ucN_G)$ and $\fM$ is any object of $\bM(\fz(\fg))$. As the endomorphisms of $\fE$ are one dimensional, we see that the functor $\fM \mapsto \fE \boxtimes \fM$ is fully faithful, as claimed.
\end{proof}
	
\subsection{The relative Steinberg and Weyl monads}\label{subsectionsteinbergweyl}
	Let us fix a parabolic subgroup $P$ with Levi factor $L$. As explained in Subsection \ref{ssrecollement}, we have a monadic adjunction:
	\[
		\xymatrixcolsep{5pc}
	\xymatrix{
		\ind^G_{P,L}: \bM(\ul)_{cusp} \ar@/_0.3pc/[r] &  \ar@/_0.3pc/[l] \bM(\ug)_{(L)}:  
		\res_{P,L}^G
	} 
	\]
	where $\bM(\ug)_{(L)}$ can be identified with the full subcategory of $\bM(\ug)$ which is generated under colimits by the essential image of $\ind^G_{P,L}$ restricted to cuspidal objects.
	
The corresponding monad is given by
\[
	\lrsubsuper{L,P}{}{\st}{P,L}{} = \res^G_{P,L} \ind^G_{P,L}:\bM(\ul)_{cusp} \to \bM(\ul)_{cusp}
	\]
	By the Mackey Formula (Proposition \ref{propositionmackey}), the functor $\lrsubsuper{L,P}{}{\st}{P,L}{}$ carries a filtration indexed by $\quot PGP$ for which the associated graded piece corresponding to $w\in \quot PGP$ is $\ind _{M \cap \ls \dw L} ^{M} \res^{\ls \dw L} _{M \cap \ls \dw L} {\dw _\ast}$. By Proposition \ref{propositionmackeycuspidal}, the restriction of the monad $\lrsubsuper{L,P}{}{\st}{P,L}{}$ to the subcategory of cuspidal objects has a filtration indexed by $W_{G,L}$ (with its poset structure defined by identification with a subset of $\quot PGP$) for which the associated graded functor is $\left(W_{G,L}\right)_\ast$. 
	
	The following key result generalizes Theorem \ref{theoremheartspringer} in the Springer block case. 
	\begin{theorem}\label{theoremheart}
		There is an equivalence of monads $\lrsubsuper{L,P}{}{\st}{P,L}{} \simeq \left(W_{G,L}\right)_\ast$ acting on $\bM(\ul)_{cusp}$.
	\end{theorem}
As in the case of the Springer block, the basic idea behind the proof of Theorem \ref{theoremheart} is to first observe that the statement holds when the functors are restricted to the regular locus, and then extend over the non-regular locus. This can be proved in a similar way to the Springer case, but it will be slightly more convenient to consider one cuspidal block at a time so we can apply the Lemmas in Subsection \ref{sec:proof-of-theorem-reftheoremheart} directly. For this to make sense, we must have that the Steinberg monad preserves the decomposition
\[
\bM(\ul)_{cusp} = \bigoplus_{(\cE)} \bM(\ul)_{(\cE)}
\]
This is equivalent to the statement that $W_{G,L}$ acts trivially on the set of cuspidal data $(\cO,\cE)$. This fact is a key feature of the generalized Springer correspondence, which we state below.

\begin{theorem}[\cite{lusztig_intersection_1984}, Theorem 9.2]\label{theoremequivariant}
\
	\begin{enumerate}
		\item Every cuspidal datum $(L,\cE)$ on $L$ is fixed under the action of $W_{G,L}$. 
		\item There is a summand of $\ind^G_{P,L}(\fE)$ which appears with multiplicity one.
	\end{enumerate}
\end{theorem}	
\begin{remark}
The results of \cite{lusztig_intersection_1984} are phrased in terms of the endomorphism algebra $A_{\fE}$ of $\ind^G_{P,L}(\fE)$. In particular, Lusztig shows in Theorem 9.2 that $A_{\fE}$ is isomorphic to the group algebra of $W_{G,L}$. We can view this statement as a combination of three separate statements:
\begin{enumerate}
\item The algebra $A_{\fE}$ is a twisted group algebra of $W_{G,L,\cE}$ (the stabilizer of the cuspidal datum $(\cO,\cE)$ in $W_{G,L}$).
\item The cuspidal datum $(\cO,\cE)$ is fixed by $W_{G,L}$ (so $W_{G,L,\cE} = W_{G,L}$).
\item The endomorphism algebra $A_{\fE}$ has a module of dimension one (it follows that $A_{\fE}$ is equivalent to the non-twisted group algebra) .
\end{enumerate}
The first statement above is a consequence of Theorem \ref{theoremheart} and can be proved directly using the techniques in this paper. The two other statements (which are equivalent to those in Theorem \ref{theoremequivariant}) require more delicate calculations, encoded in the dimension estimates appearing in the proof of \cite{lusztig_intersection_1984}, Theorem 9.2; we will not attempt to give a self-contained treatment of these results in the present paper.
\end{remark} 
 	\begin{remark}
		In fact, Theorem 9.2 in \cite{lusztig_intersection_1984} gives more.	
Let $P_1, \ldots , P_r$ be a complete list of minimal parabolic subgroups of $G$ properly containing $P$, and $L_1,\ldots, L_r$ corresponding Levi factors containing $L$. Then $W_{L_i,L}$ is a subgroup of $W_{G,L}$ of order at most 2 (it acts faithfully on the 1-dimensional vector space $\fz(\fl)/\fz(\fl_i)$ by reflections). Let $s_i$ denote the generator of $W_{L_i,L}$. The $s_i$ generate $W_{G,L}$, which acts by reflections on $\fz(\fl)/\fz(\fg)$.
In particular, if $L$ is a Levi subgroup of $G$ which admits a cuspidal local system on a nilpotent $L$-orbit, then all the subgroups $W_{L,i,L}$ are non-trivial (and thus of order 2). In this case $W_{G,L}$ is the Weyl group of a certain root system on $\fz(\fl)/\fz(\fg)$ (see \cite{lusztig_coxeter_1976} 5.9). This is a very strong restriction on which Levi subgroups of $G$ admit cuspidals. For example, the Levi subgroup $L_p$ of $SL_n$ consisting of block diagonal matrices corresponding to a partition $n = p_1 + \ldots + p_k$ can only admit a cuspidal when $p_1=\ldots =p_k$ for some divisor $p_k$ of $n$, in which case $W_{G,L_p} = S_k$ acting by permuting the blocks of a matrix (and there are in fact cuspidals in that case see \cite{lusztig_intersection_1984} 10.3).
	\end{remark}
	\begin{remark}
		It is a result of Bala-Carter theory that if $\cO$ is a distinguished nilpotent $L$-orbit, then the $G$-saturation of $\cO$ intersects $\fl$ precisely at $\cO$ (see \cite{sommers_generalization_1998}). In particular $\cO$ is preserved by the action of $W_{G,L}$. The first part of Theorem \ref{theoremequivariant} can be thought of as a generalization of this fact, where the set of distinguished orbits is refined by the set of cuspidal data.
	\end{remark}
\begin{proof}[Proof of Theorem \ref{theoremheart}]
By Theorem \ref{theoremequivariant} and the Mackey formula, parabolic induction and restriction restrict to a monadic adjunction for each cuspidal datum $(L,\cE)$:
\[
	\xymatrixcolsep{5pc}
	\xymatrix{
		\ind^G_{P,L}: \bM(\ul)_{(\cE)} \ar@/_0.3pc/[r] &  \ar@/_0.3pc/[l] \bM(\ug)_{(L,\cE)}:  
		\res_{P,L}^G
	} 
	\]
Let $\st_{(\cE)}$ denote the restriction of the Steinberg monad to the block $\bM(\ul)_{(\cE)}$, which is equivalent to $\bM(\fz(\fl))$. Thus we can identify endofunctors of the block with $\bM(\fz(\fl)\times \fz(\fl))$. Just as for Theorem \ref{theoremheartspringer}, the result then follows from the three lemmas in Subsection \ref{sec:proof-of-theorem-reftheoremheart}, after replacing $\fh$ by $\fz(\fl)$, and $W$ by $W_{G,L}$.
\end{proof}	

\begin{corollary}\label{corollaryequiv}
Every simple cuspidal object $\fE$ in $\bM(\ucN_L)$ carries a $W_{G,L}$-equivariant structure. In particular there is a $W_{G,L}$-equivariant equivalence:
$
\bM(\ul)_{(\cE)} \simeq \bM(\fz(\fl))
$
\end{corollary}
\begin{proof}
According to Theorem \ref{theoremequivariant}, $\fE$ is the parabolic restriction of an object in $\bM(\ucN_G)$ (namely the simple summand of $\ind^G_{P,L}(\fE)$ appearing with multiplicity one). This it carries the structure of a $\lrsub{L,P}{\st}{P,L}$-module. This is the same as a $W_{G,L}$-equivariant structure by Theorem \ref{theoremheart}.
\end{proof}

	\subsection{Proof of the main results}\label{ssmain}
	
	In this subsection we will deduce Theorem \ref{theoremabelian} and Theorem \ref{maintheoremindres} from Theorem \ref{theoremheart}. First we note:
	\begin{corollary}
		For each cuspidal datum $(L,\cE)$, there is an equivalence of categories:
		\[
		\bM(\ug)_{(L,\cE)} \simeq \bM(\fz(\fl))^{W_{G,L}}.
		\]
	\end{corollary}
	\begin{proof}
		This is an application of the Barr-Beck theorem (see Appendix \ref{appendixbb}). By construction, the functor 
		\[
		\res^G_{P,L}: \bM(\ug)_{(L,\cE)} \to \bM(\ul)_{(\cE)}
		\]
		is monadic, and by Theorem \ref{theoremheart}, the corresponding monad is given by $\left(W_{G,L}\right)_\ast$ acting on $\bM(\ul)_{(\cE)}$ (which is $W_{G,L}$-equivariantly identified with $\bM(\fz(\fl))$ by Corollary \ref{corollaryequiv}).
	\end{proof}

By Proposition \ref{propositiongroup} and Remark \ref{remarkgroup}, we obtain:
\begin{lemma}\label{lemmasummand}
	For every object $\fM \in \bM(\ug)_{(L,\cE)}$, we have
	\[
	\fM \simeq \left(\ind^G_{P,L}\res^G_{P,L} (\fM)\right)^{W_{G,L}}.
	\]
	In particular, $\fM$ is a direct summand of $\ind^G_{L}\res^G_{L} (\fM)$.
\end{lemma}
Thus the category $\bM(\ug)_{(L,\cE)}$ is just the Karoubian completion of the essential image of $\ind^G_{L}$ on $\bM(\ul)_{(\cE)}$ (i.e. the subcategory consisting of direct summands of parabolic induction from $\bM(\ul)_{(\cE)}$).

\begin{remark}
	The corresponding statement is not true for the derived blocks $\bD(\ug)_{(L,\cE)}$: not every such object is a direct summand of a parabolic induction. See \cite{gunningham_derived_2017} for further details. 
\end{remark}

	Now we show that parabolic induction of cuspidal objects is independent of the choice of parabolic (later we will drop the cuspidal condition). 
	\begin{corollary}\label{corollaryindependent}
	Suppose $P$ and $P'$ are two parabolic subgroups which contain $L$ as a Levi factor. 
		Then there is a canonical natural isomorphism $\ind^G_{P,L} \simeq \ind^G_{P',L}$ after restricting to $\bM(\ul)_{(\cE)}$.
	\end{corollary}
	\begin{proof}
		By adjunction, the data of a natural transformation $\ind^G_{P,L} \to \ind^G_{P',L}$ is equivalent to that of a natural transformation
		\begin{equation}\label{equationnattrans}
		Id_{\bM(\ul)_{(\cE)}} \to \lrsubsuper{L,P}{}{\st}{P',L}{}
		\end{equation}
		By the same argument as in the proof of Theorem \ref{theoremheart}, there is a canonical natural isomorphism $\lrsubsuper{L,P}{}{\st}{P',L}{} \simeq \left(W_{G,L}\right)_\ast$. This isomorphism is constructed by first observing that such an isomorphism exists once we restrict to the regular locus (where neither functor depends on the choice of parabolic), and there is a unique extension over the non-regular locus. The inclusion of the identity $1 \in W_{G,L}$ defines the required natural transformation (\ref{equationnattrans}).
	\end{proof} 
By the uniqueness of adjoints, it follows that parabolic restriction $\res^G_{P,L}$ is also independent of the choice of parabolic when restricted to $\bM(\ug)_{(L)}$ (we may drop this last assumption soon).

		\begin{corollary}\label{corollaryorthogonal}
			The subcategories $\bM(\ug)_{(L,\cE)}$ and $\bM(\ug)_{(M,\cF)}$ are orthogonal if $(L,\cE)$ and $(M,\cF)$ are non-conjugate cuspidal data.
		\end{corollary}
		\begin{proof}
			By Proposition \ref{propositionsplit}, if $M$ is not conjugate to $L$ then the essential image of $\ind^G_{P,L}|_{cusp}$ is orthogonal to the essential image of $\ind^G_{Q,M}|_{cusp}$. Thus it follows that any direct summand of such objects must also be orthogonal, so by Lemma \ref{lemmasummand}, $\bM(\ug)_{(L)}$ is orthogonal to $\bM(\ug)_{(M)}$. It remains to consider the case when $M=L$, but $\cE$ is not conjugate to $\cF$. Recall that by the results of Subsection \ref{sscuspidal}, $\bM(\ul)_{(\cE)}$ is orthogonal to $\bM(\ul)_{(\cF)}$, and these blocks are preserved by the action of $W_{G,L}$ by Theorem \ref{theoremequivariant}. Applying the Mackey formula as in the proof of Proposition \ref{propositionsplit} gives the result.
	\end{proof}	
	
At this point, the proof of Theorems \ref{theoremabelian}, \ref{maintheoremindres}, and \ref{maintheorempartition} is just a matter of putting together the pieces.
	
	\begin{proof}[Proof of Theorem \ref{theoremabelian}]
	We have shown that the subcategories, $\bM(\ug)_{(L,\cE)}$ are pairwise orthogonal, and each is equivalent to $\bM(\fz(\fl))^{W_{G,L}}$ (for the relevant $L$). By the recollement situation of \ref{ssrecollement}, every object of $\bM(\ug)$ is an iterated extension of objects of $\bM(\ug)_{(L,\cE)}$; we now know that these extensions must all split. We have proved Theorem \ref{theoremabelian}:
\begin{equation}
\label{equationdecomposition}
\bM(\ug) = \bigoplus_{(L,\cE)} \bM(\ug)_{(L,\cE)} \simeq \bigoplus_{(L,\cE)} \bM(\fz(\fl))^{W_{G,L}}
\end{equation}
\end{proof}


It remains to clear up a few points in Theorem \ref{maintheoremindres}, namely, the dependence on the parabolic and the splitting of the Mackey filtration. These points have been addressed (by Corollary \ref{corollaryindependent} and Theorem \ref{theoremheart}) in the case of parabolic induction and restriction on a single cuspidal block. Thus, we must look at the functor of induction and restriction on all blocks at once.

Consider the functor $\res^G_{Q,M}$ for a fixed parabolic subgroup $Q$ with Levi factor $M$. If $L$ is not conjugate to a subgroup of $M$, then by the Mackey formula, $\res^G_{Q,M}$ kills the summand $\bM(\ug)_{(L,\fE)}$. On the other hand, if $L \subseteq M$, and $P \subseteq Q$, then $(L,\cE)$ can also be considered as a cuspidal datum for $M$. The compatibility of Proposition \ref{propositionindresfactor} implies that the functor $\res^G_{Q,M}$ restricts to a functor from $\bM(\ug)_{(L,\cE)}$ to $\bM(\um)_{(L,\cE)}$. Under the equivalence of (\ref{equationdecomposition}), this corresponds to the forgetful functor:
\[
\xymatrix{
	\Upsilon^{W_{G,L}}_{W_{M,L}}:\bM(\fz(\fl))^{W_{G,L}} \ar[r] & \bM(\fz(\fl))^{W_{M,L}}
}
	\] 
A similar statement holds for $\ind^G_{Q,M}$, which corresponds to the induction functor:
\[
\xymatrix{
\Gamma^{W_{G,L}}_{W_{M,L}}:\bM(\fz(\fl))^{W_{M,L}} \ar[r] & \bM(\fz(\fl))^{G_{M,L}}
}
\] 
Putting this all together we have the following commutative diagram (when read from top left to bottom right or from top right to bottom left): 
	\[
		\xymatrixcolsep{5pc}
		\xymatrixrowsep{3pc}
		\xymatrix{
			\bM(\ug) \ar@/^/[r]^-{\res^{G}_{Q,M}} & \ar@/^/[l]^-{\ind^{G}_{Q,M}} \bM(\um) \\
			\bigoplus\limits_{\substack{(L,\cE)\\L\subseteq G}}^\perp \bM(\fz(\fl))^{W_{G,L}}
			\ar@{<->}[u]^-\wr  \ar@/^/[r] & \ar@/^/[l] \bigoplus\limits_{\substack{(L,\cE)\\L\subseteq M}}^\perp \bM(\fz(\fl))^{W_{M,L}} \ar@{<->}[u]_-\wr
		}
		\]
	
The remaining parts of Theorem \ref{maintheoremindres} now follow: the Mackey formula must split (as it is equivalent to the Mackey formula corresponding to induction and restriction for finite groups) and the functors parabolic induction and restriction are independent of the choice of parabolic (as they are on each block). 

\subsection{Singular support and parabolic restriction}\label{ss:ss}

The goal of this subsection is to prove the following result (which completes the proof of Theorem C).

\begin{remark}
	We fix a non-degenerate $G$-invariant symmetric bilinear form on $\fg$. Given an object $\fM \in \bM(\ug)$ its singular support $SS(\fM)$ is considered as a $G$-stable closed subset of the commuting variety $\comm(\fg) \subseteq \fg\times \fg$ which is contained in the commuting variety $\comm(\fg)$. It is conic with respect to the scaling action on the second factor. Given $x\in \fg$ we write $SS(\fM)_x$ for $\{y\in \fg \mid (x,y) \in SS(\fM)\}$. Note that this is non-empty if and only if $x\in Supp(\fG)$.
\end{remark}

\begin{theorem}\label{thm:ss}
	Let $\fG \in \bM(\ug)_{(L)}$ be a non-zero object. Then we have
	\begin{enumerate}
		\item 
		\[
		SS(\fG) \subseteq \comm(\fg)_{\geq (L)}
		\]
		\item
		\[
		SS(\fG) \cap \comm(\fg)_{\leq (L)} \neq \emptyset
		\]
	\end{enumerate}
\end{theorem}

Recall that the condition $\fG\in \bM(\ug)_{(L)}$ means that $\fG = \ind^G_L(\fL)^{W_G,L}$ for some cuspidal object $\fL \in \bM(\ul)_{cusp}$ (we also have $\fL \cong \res^G_L(\fG)$). The first part of the theorem is relatively straightforward. It is a consequence of the standard bounds on singular support with respect to a smooth pullback and proper pushforward (see e.g. \cite{kashiwara_d-modules_2003}[Theorem 4.7, Theorem 4.27]). Indeed, parabolic induction is a composite of such functors, and by \cite{gunningham_generalized_2018}[Lemma 3.8] we have $SS(\fL) \subseteq \comm(\fl)_\heartsuit$ from which the required bound on $SS(\fG)$ follows (using that $\fG$ is a direct summand of $\ind^G_L(\fL)$). Alternatively, it follows from the ``if'' direction of the proof of Lemma 3.19 in \cite{gunningham_generalized_2018} (which does not make use of the erroneous results). Indeed, by assumption we have $\res^G_M(\fG)\cong 0$ for all Levis $M$ with $(M) \ngeq (L)$. Then \emph{loc. cit.} states that 
\[
SS(\fM) \subseteq \bigcap_{(M)\ngeq(L)} \comm(\fg)_{\nleq(M)} = \comm(\fg)_{\geq(L)}
\]

The second part of Theorem \ref{thm:ss} is more tricky, as singular support is hard to control under parabolic restriction. The remainder of the subsection is devoted to its proof. Let us first note that singular support is preserved by parabolic restriction once we restrict to the regular locus:
\begin{lemma}\label{lem:noncharacteristic}
	Let  be as above, and take $x\in \fl^\reg$. Then
	\[
	SS(\res^G_L(\fG))_x = SS(\fG)_x.
	\]
\end{lemma}
\begin{proof}
	Recall that $\res^G_L(\fG)|_{\ul^\reg}$ is identified with $(d^\reg)^!\fG$ where $d^\reg:\ul^\reg \to \ug$ is an \'etale morphism. The lemma follows from the fact that singular support is preserved by \'etale (or more generally, smooth) morphisms. Alternatively, in non-stacky terms the result follows from the fact that the inclusion $\fl^\reg \hookrightarrow \fg$ is non-characteristic with respect to $G$-equivariant $D$-modules on $\fg$.
\end{proof}

The following key lemma relates the singular support of a nilpotent orbital $D$-module and its Fourier transform.

\begin{lemma}\label{lem:nilp}
	Suppose $\fG \in \bM(\ug)$ is supported on the nilpotent cone and $\fL := \res^G_L(\fG)$ is non-zero. Then
	\[
	SS(\fG)_0 \cap \fg_{\leq (L)} \neq \emptyset
	\]
\end{lemma}
\begin{proof}
	Note that $\fL$ is supported in the nilpotent cone $\cN_L$ (and in particular, is $\G_m$-monodromic and regular holonomic). Consider the decomposition $\fl = \fz(\fl) \oplus [\fl,\fl]$. As $\fL$ is supported on $\cN_L \subseteq [\fl,\fl]$, its Fourier transform takes the form 
	\[
	\F_\fl(\fL) \cong \cO_{\fz(\fl)} \boxtimes \F_{[\fl,\fl]}(\fL)
	\]
	where in the second factor we consider $\fL$ as a $D$-module on $[\fl,\fl]$. In particular, any $y\in \fz(\fl)^\reg$ is contained in $Supp(\cO_{\fz(\fl)})$ and thus in $Supp(\F_{\fl} \fL)$ (note that $\F_{[\fl,\fl]}\fL$ is $\G_m$-monodromic, so $0$ is contained in its support).
	
	Now recall that Fourier transform commutes with parabolic restriction (see e.g.  \cite{gunningham_generalized_2018}[Lemma 3.7]). Thus we can apply Lemma \ref{lem:noncharacteristic} to $\F_\fg(\fG)$ and $\F_\fl(\fL)$ to deduce that 
	\[
	SS(\F_{\fg}\fG)_y = SS(\res^G_L \F_{\fg}\fG) = SS(\F_{\fl}\fL)_y
	\]
	By assumption, the right hand side is a non-empty conical closed subset of $\fg$, and thus contains the element $0$. Finally, by Proposition \ref{prop:brylinski}, we have that $(0,y) \in SS(\fG) \cap \fg_{\leq (L)}$ as required.
\end{proof}

\begin{proof}[Proof of Theorem \ref{thm:ss}]
	It remains to prove part 2 of the theorem. By assumption we have $\fG = \ind^G_L(\fL)^{W_{G,L}}$ for some cuspidal $\fL \in \bM(\ul)_{cusp}$. Recall that $\fL$ is supported in $\fl_\heartsuit = \fz(\fl) \times \cN_L$. We must find $(x,y) \in SS(\fG)$ with $H_G(x,y) \subseteq L$ (or equivalently, $H_G(x,y)=L$). 
	
	If $Supp(\fL)$ intersects $\fl^\reg$ non-trivially, we are done by Lemma \ref{lem:noncharacteristic}. On the other hand if $Supp(\fL) \subseteq \cN_L$ we are done by Lemma \ref{lem:nilp}. In general, we consider the locally closed partition of $fz(\fl) \times \cN_L$ given by $\fz(\fm)^\reg \times \cN_L$ as $\fm$ ranges over the (finitely many) Levi subalgebras of $\fg$ containing $\fl$. By Lemma \ref{lem:stratifications} we can find such a Levi subalgebra $\fm$ such that $\fL' := R^0\Gamma_{\fz(\fm) \times \cN_L} \fL$ is a non-zero submodule of $\fL$ and there exists $x\in Supp(\fL')\cap \fz(\fm)^\reg \times \cN_L$. Let $M$ denote the Levi subgroup of $G$ corresponding to $\fm$ and consider the object 
	\[
	\fM := \res^G_M(\fG) \cong \ind^M_L(\fL)^{W_{M,L}} \in \bM(\um)
	\]
	As $\fz(\fm) \times \cN_L$ is $W_{M,L}$-stable, $\fL'$ is still $W_{M,L}$-equivariant (though not in general $W_{G,L}$-equivariant) and we obtain a non-zero submodule
	\[
	\fM' = \ind^M_L(\fL)^{W_{M,L}} \hookrightarrow \fM
	\]
	with $x\in Supp(\fM') \cap \fm^\reg$. We have now reduced to finding $y\in SS(\fM')_x$ with $H_M(y)\subseteq L$ (in that case, by Lemma \ref{lem:noncharacteristic}, $y\in SS(\fG)_x$ and $H_G(x,y) = H_G(x) \cap H_G(y) = G_M(y) \subseteq L$ as required). 
	
	Let us split off the copy of $\fz(\fm)$ from $\fl$ using our fixed invariant form; that is, we consider an orthogonal decomposition:
	\[
	\fl = \fz(\fm) \oplus \fl'
	\]
	We have a corresponding decomposition $\fm = \fz(\fm) \oplus \fm'$. Note that $\fl'$ can be considered as a Levi subalgebra of the Lie algebra $\fm' = [\fm,\fm]$.
	Recall that $\fL'$ can be written as external tensor products with respect to this decomposition
	\[
	\fL' \cong \fL'_s \boxtimes \fL'_n
	\]
	where $\fL'_n$ is supported on the nilpotent cone of $\fl'$ (see e.g. \cite{gunningham_generalized_2018}[Proposition 4.21]). As in the proof of \cite{gunningham_generalized_2018}[Lemma 4.19], we have that parabolic induction and restriction between $L$ and $M$ acts as the identity on the $\fz(\fm)$ factor. Thus $\fM$ is also an external tensor product (with the first factor unchanged)
	\[
	\fM' \cong \fL'_s  \boxtimes \fM'_n
	\] 
	Applying Lemma \ref{lem:nilp} (with $G$ replaced by $M$ and $\fG$ replaced by $\fM'_n$) we deduce that there exists $y\in SS(\fM'_n)_0$ with $H_M(y) \subseteq L$. We claim that $(x,y) \in SS(\fM')$. Indeed, note that $SS(\fM')$ can be identified with $SS(\fL'_s) \times SS(\fM'_n)$, and under this identification, the element $(x,y)$ corresponds to $(x,0)$ in the first factor and $(0,y)$ in the second. Thus we have found the required element $y \in SS(\fM')_x$ completing the proof of Theorem \ref{thm:ss}.
\end{proof}

\appendix

\section{Monads and Recollement situations}\label{appendixcat}
We review the Barr-Beck theorem in the context of Grothendieck abelian categories and explain how certain adjunctions give rise to recollement situations. The case when the monad is given by the action of a finite group is of particular interest to us. The proof of the Barr-Beck theorem can be found in in \cite{barr_toposes_1985}, and the background material on abelian categories in \cite{popescu_abelian_1973}. A modern treatment of the theory of recollement situations for abelian categories can be found in \cite{franjou_comparison_2004}. We will only give brief sketches of the proofs here.

Throughout this appendix, we maintain the following:
	\begin{assumptions}\label{assumptions}
	Let $\cC$ and $\cD$ be Grothendieck abelian categories,  
	$
	F:\cD \to \cC
	$
	is an functor with a left adjoint $F^L$. We assume that both $F$ and $F^L$ are exact and preserve direct sums (it follows from the adjoint functor theorem that $F$ also has a right adjoint $F^R$).
	\end{assumptions}

	\subsection{The Barr-Beck theorem and recollement situations}\label{appendixbb}
	Let $T=FF^L$ denote the corresponding monad acting on $\cC$. We denote by $\cC^T$ the category of $T$-modules (also known as $T$-algebras) in $\cC$. Note that for any object $d\in \cD$, $F(d)$ is a module for $T$. Thus we have the following diagram:
	\[
	\xymatrix{
		\cD \ar[rr]^F \ar[rd]_{\widetilde{F}} & & \cC\\
		& \cC^T \ar[ru] &
	}
	\]
	\begin{definition}
		A functor $F: \cD \to \cC$ is called \emph{conservative} if whenever $F(x) \simeq 0$
		then $x \simeq 0$.\footnote{The usual definition of a conservative functor is a functor $F$
			such that if $F(\phi)$ is an isomorphism, then $\phi$ is an isomorphism. This definition is
			equivalent to the one above, in our context, by considering the the kernel and cokernel
			of $\phi$.}
	\end{definition}
	\begin{theorem}[Barr-Beck \cite{barr_toposes_1985}]\label{theorembarrbeck}
		The functor $\widetilde{F}:\cD \to \cC^T$ has a fully faithful left adjoint, $J$. If $F$ is conservative, then $\widetilde{F}$ and $J$ are inverse equivalences.
	\end{theorem}
	There is an explicit formula for the left adjoint, as follows. Given an object $c \in \cC^T$, consider the diagram
		\begin{equation}\label{diagramsimplicial}
			F^LFF^Lc \rightrightarrows	F^Lc 
		\end{equation}
		where one of the maps is given by the $T$-module structure on $c$, and the other is given by the counit map of the adjunction. Then $J(c)$ is defined to be the coequalizer of this diagram, or in other words, the cokernel of the map given by the difference of the two maps above.
		One can check that $J$ is left adjoint to $\widetilde{F}$ and that the unit $ c \xrightarrow{\sim} \widetilde{F}J(c)$, so that $J$ is fully faithful as required.

	\subsection{Adjunctions and Recollement situations}\label{appendixrecollement}
	We further study the case when the functor $F$ is not necessarily conservative. 
	Let $\cK$ denote the kernel of $F$, i.e. the full subcategory of $\cD$ consisting of objects $d$ such that $F(d) \simeq 0$. Let $\cQ$ denote the quotient category $\cD/\cK$, which is the localization of $\cD$ with respect to the multiplicative system of morphisms that are taken to isomorphisms under the functor $F$. The subcategory $\cK$ is automatically localizing, and thus the quotient morphism $\fj^\ast: \cD \to \cQ$ has a fully faithful \emph{right} adjoint, denoted $\fj_\ast$. (see e.g. \cite{popescu_abelian_1973}). On the other hand, $F$ descends to a conservative functor $F_\cQ:\cQ \to \cC$. Moreover, one readily checks that the left adjoint to $F_\cQ$ is given by the composite $\fj^\ast F^L$, and the resulting adjunction $\cQ \leftrightarrows \cC$ affords the same monad $T$ on $\cC$. Thus, by Theorem \ref{theorembarrbeck} (Barr-Beck), we can identify $\cQ$ with the category of $T$-modules, $\cC^T$. Using this identification, the bar construction defines a fully faithful \emph{left} adjoint to $\fj^\ast$, which we denote $\fj_!$. 
	
	Let $\ffi_\ast: \cK \to \cD$ denote the embedding.
	Consider the functor $\fk:\cD \to \cD$ given by the cokernel of the unit map $Id_\cD\to \fj_!\fj^\ast$. The essential image of the functor $\fk$ is contained in $\cK$, and thus we can write $\fk = \ffi_\ast \ffi^\ast$ for where $\ffi^\ast: \cD \to \cK$ is left adjoint to $\ffi_\ast$. Similarly, the kernel of the counit map $\fj_\ast \fj^\ast \to Id_\cD$ is of the form $\ffi_\ast\ffi^!$, where $\ffi^!$ is right adjoint to $\ffi_\ast$.

	These facts are summarized as follows:
	
	\begin{theorem}\label{theoremcoloc}
		There is a recollement situation:
	\[
	\xymatrixcolsep{7pc}
	
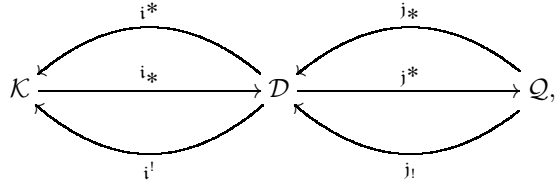
\xymatrix{
		\cK \ar[r]^{\ffi_{\ast}} & \ar@/_2pc/[l]_{\ffi^\ast}\ar@/^2pc/[l]^{\ffi^!} \cD \ar[r]^{\fj^\ast} & \ar@/^2pc/[l]^{\fj_{!}} \ar@/_2pc/[l]_{\fj_{\ast}} \cQ,
	}
	\]
		\end{theorem}
	In particular:
		\begin{itemize}
			\item The functor $\fj^\ast$ is left adjoint to $\fj_\ast$, and right adjoint to $\fj_!$.
			\item The functor $\ffi_\ast$ is left adjoint to $\ffi^!$, and right adjoint to $\ffi^\ast$.
			\item The functors $\fj_\ast$, $\fj_!$, and $\ffi_\ast$ are fully faithful;
			\item There are exact sequences of functors
			\begin{align*}
			0 \to \ffi_\ast \ffi^! \to Id_\cD \to \fj_\ast \fj^\ast\\
			\fj_! \fj^\ast \to Id_\cD \to \ffi_\ast \ffi^\ast \to  0
			\end{align*}
				\end{itemize}
			Also note that, by the construction of the left adjoint $\fj_!$, the essential image of $\cQ$ 			in $\fj_!$ consists of quotients of objects in the essential image of $F^L$.

	\begin{remark}
		If the categories $\cC$ and $\cD$ are compactly generated, and the functor $F$, in addition, takes compact objects in $\cD$ to compact objects in $\cC$, then the right adjoint $F^R$ preserves direct sums. In that case, the recollement situation of Theorem \ref{theoremcoloc} restricts to one on the level of small categories of compact objects (this is the case for the main example in this paper).
	\end{remark}

	\subsection{Monads and finite group actions}\label{subsectionfinite}
	Suppose additionally that a finite group $W$ acts on $\cC$. This means that there are functors
	\[
	w_\ast: \cC \to \cC,
	\]
	together with natural isomorphisms $\phi_{w,v}: w_\ast v_\ast \xrightarrow{\sim} (wv)_{\ast}$ satisfying the natural cocycle condition. There is an associated monad $W_\ast$ acting on $\cC$ given by the formula:
	\[
	W_\ast (c) = \bigoplus _{w\in W} w_\ast (c).
	\]
	The category $\cC^W$ of $W$-equivariant objects, is equivalent to the category of modules for the monad $W_\ast$. 
	
	\begin{remark}
		The functor $W_\ast$ also admits the structure of a comonad acting on $\cC$ (and $W$-equivariant objects in $\cC$ can also be identified with comodules for this comonad). Considering this monad and comonad structure together, one arrives at the notion of Frobenius monad.
	\end{remark}
	
	\begin{proposition}\label{propositiongroup}
		Suppose there is an isomorphism of monads $T \simeq W_\ast$. Given an object $c \in \cC^W$, $W$ acts by automorphisms on the object $F^L(c)$ such that the object $\fj_!(c)$ is given by the coinvariants $F^L(c)_W$ of this action.
	\end{proposition}
	\begin{proof}
		The $W$ action comes from the identities:
		\[
		\End(F^L(c)) \simeq \bbHom(c, FF^L(c)) \simeq \bbHom(c, W_\ast c) \simeq \Z[W] \otimes \End(c).
		\]
		The identification of $\fj_!(c)$ with the coinvariants is by inspection of the formula \ref{diagramsimplicial} for $\fj_!$ as a coequalizer.
	\end{proof}
	
	\begin{remark} \label{remarkgroup}
		\begin{enumerate}
			\item Note that the functor of coinvariants for a finite group action is exact (and agrees with the functor of invariants via the norm map). In particular, $\fj_!(c)$ is a direct summand of $F^L(c)$. 
			\item If $\cC$ and $\cD$ are $\C$-linear abelian categories and $c \in \cC^T$ is a simple object, then $\End(F^L(c)) \simeq \C[W]$.
		\end{enumerate}
	\end{remark}
	
	\subsection{Filtrations}
The following notation will be useful for us later.

\begin{definition}\label{definitionfiltration}
	A \emph{filtration} of an object $a$ in a category $\cC$, indexed by a poset $(I,\leq)$, is a functor 
	\begin{align*}
	(I,\leq)              & \to \cC/a   \\,
	i \mapsto (a_{\leq i} & \mapsto a). 
	\end{align*}
	In the cases of interest to us, $I$ will be a finite poset with a maximal element $i_{max}$, and we demand in addition that $a_{\leq i_{max}}\to a$ is an isomorphism. 
	\end{definition}
	
	We will apply this definition in two settings: either $\cC$ is a Grothendieck abelian category as in the previous subsections, or $\cC$ is a triangulated category (or stable $\infty$-category).	
	In the abelian category setting we will ask for the structure maps to be monomorphisms, but in the triangulated/stable setting, we have no such condition.
	
	Let $a_{<i}$ denote the colimit of $a_{\leq j}$ over $j<i$.
	For any $i\in I$, we set $a_i$ to be the cokernel (or cone) of $a_{<i} \to a_{\leq i}$. Thus, we think of the object $a$ as being built from $a_i$ by a sequence of extensions. The \emph{associated graded} object is defined to be $\bigoplus_{i\in I} a_i$.

	\section{$D$-modules}\label{appendixdmod}
	There is a vast literature on the theory of $D$-modules; a good elementary reference is the book of Hotta--Takeuchi--Tanisaki \cite{hotta_d_2008}. 
	The triangulated category of equivariant $D$-modules was defined by Bernstein--Lunts \cite{bernstein_equivariant_1994} (in the context of sheaves) and Beilinson--Drinfeld \cite{beilinson_quantization_}, and $\infty$-categorical enhancements have been considered in \cite{gaitsgory_crystals_2011} \cite{ben-zvi_character_2009} (see also the recent book project of Gaitsgory-Rozenblyum \cite{gaitsgory_study_2017}). Below we outline some of the key properties of the theory that we will need for this paper.
	
	\subsection{On smooth affine varieties}\label{appendixdmodvar}
	Suppose $U$ is a smooth affine\footnote{We restrict to affine varieties only for convenience; all the material in this appendix (except the results relating specifically to vector spaces) may be generalized to general smooth varieties and Artin stacks.} algebraic variety, and let $\fD_U$ denote the ring of algebraic differential operators on $U$. We will write $\bM(U)$ for the category of $\fD_U$-modules. An object of $\bM(U)$ is called coherent if it is finitely generated (or equivalently, finitely presented), and we denote the subcategory of such objects by $\bM_{coh}(U)$. 
	
	We denote by $\bD(U)$, the (unbounded) derived category of $\bM(U)$. Really, I will want to consider the corresponding stable $\infty$-category, but for the purposes of this paper we will only need to consider its homotopy category, which is the usual, triangulated, derived category. The compact objects of $\bD(U)$ are given by perfect complexes of $\fD_U$-modules (or equivalently, bounded complexes whose cohomology objects are finitely generated $\fD_U$-modules, as $\fD_U$ has finite homological dimension); such complexes are called \emph{coherent}. 
	
	Given a morphism of smooth affine algebraic varieties 
	\[
	f:U\to V
	\]
	we have functors:
	\begin{align*}
	f^\circ: \bD(V) \to \bD(U) \\
	f_\ast : \bD(U) \to \bD(V).
	\end{align*}
	These are defined by derived tensor product (with no cohomological shift) with the \emph{transfer bimodules} $\fD_{U\to V}$ and $\fD_{V\leftarrow U}$ (see \cite{hotta_d_2008}, Section 1.3 for the definitions).
	If $f$ is smooth, then $f^\circ$ is $t$-exact and preserves coherence; if $f$ is a closed embedding, then $f_\ast$ is $t$-exact and preserves coherent objects \cite[Sections 2.4 and 2.5]{hotta_d_2008}.
	
	\subsection{On quotient stacks}\label{appendixdmodquot}
	Let $K$ be an affine algebraic group acting on the smooth variety $U$ and let $\fk$ denote the Lie algebra of $K$. There is a morphism of algebras
	\[
	\mu^\ast:\cU(\fk) \to \fD_{U}
	\]
	(where $\cU(\fk)$ is the universal enveloping algebra of $\fk$) representing the infinitesimal action. Given an object $\fM$ of $\bM(U)$ (i.e. a $\fD_U$-module), a $K$-equivariant structure on $\fM$ is an action of $K$ on $\fM$ by $\fD_U$-module morphisms, such that the corresponding action of $\fk$ agrees with the one coming from $\mu^\ast$. We write $\bM(U)^K$ or $\bM(X)$ for the abelian category of $K$-equivariant $\fD_U$-modules (where $X = U/K$ is the quotient stack). 
	
	Note that if $K$ is connected, then $K$-equivariance is actually a \emph{condition}: the $K$-action (when it exists) is determined by $\mu^\ast$. Thus $\bM(U)^K$ embeds as a full subcategory of $\bM(U)$ in that case.
	
	It turns out that the ``correct'' definition of $\bD(X) = \bD(U)^K$ does not agree with the derived category of $\bM(X)$ in general. Definitions and properties of the equivariant derived category $\bD(X)$ (in varying degrees of generality) can be found in \cite{beilinson_quantization_,ben-zvi_character_2009,gaitsgory_study_2017}.\footnote{One approach is to use the theory of stable $\infty$-categories (developed by Lurie in \cite{lurie_stable_2006,lurie_higher_2011}), which can be thought of as an enhancement of theory of triangulated categories (the homotopy category of a stable $\infty$-category is a triangulated category). In this enhanced setting, $\bD(X)$ can be defined as the limit (in some appropriate $\infty$-category of stable $\infty$-categories) of the cosimplicial diagram of categories obtained from the \v{C}ech simplicial object associated to the cover $U\to X$ (this is a homotopical formulation of the notion of descent). In this paper we will not need to consider the $\infty$-categorical enhancements, only their underlying triangulated category.}
This is a triangulated category which carries a $t$-structure, whose heart is $\bM(X)$. The subcategory $\bD^b_{coh}(X)$ consists of complexes with finitely many non-zero cohomology groups, each of which are coherent (as $D$-modules on $Y$). The category $\bD(X)$ is compactly generated and the subcategory $\bD_{com}(X)$ of compact objects is contained in $\bD^b_{coh}(X)$ (these subcategories agree when $X$ is a scheme, but may differ if $X$ is a stack).

	\subsection{Functors}\label{appendixdmodfun}
	Suppose $U$ and $V$ are smooth varieties, $K$ is an algebraic group acting on $U$, $L$ is an algebraic group acting on $V$, and we fix a homomorphism $\phi: K \to L$. A morphism $\widetilde{f}:U \to V$ which is equivariant for the group actions (via $\phi$) gives rise to a morphism of quotient stacks $f: U/K \to V/L$. This morphism is \emph{representable} if $\phi$ is injective, and \emph{safe} if the kernel of $\phi$ is unipotent.
	
	\begin{remark}
		The term \emph{safe} was defined more generally in the paper \cite{drinfeld_some_2013} of Drinfeld-Gaitsgory. It is shown in that paper that safe morphisms of stacks give rise to a well-behaved pair of functors on $D$-modules. 
	\end{remark}
	
	The morphism $f$ of quotient stacks is \emph{smooth} if and only if the corresponding morphism $\widetilde{f}:U\to V$ is smooth; in that case, the \emph{relative dimension of $f$} is given by the relative dimension of $\widetilde{f}$ minus the dimension of the kernel of $K\to L$. The morphism $f$ is called proper if it is representable and $\widetilde{f}$ is proper.
	
	\begin{example}
		Let $B$ be a Borel subgroup of a reductive algebraic group, $U$ the unipotent radical of $B$, and $H=B/U$. Then the morphism
		\[
		B\adjquot B \to H\adjquot H
		\]
		is safe and smooth of relative dimension $0$. Similarly, $U\adjquot U$ is safe and smooth of relative dimension $0$ over a point. On the other hand, the morphism $pt/H \to pt$ is not safe. The morphism $B/B \to G/G$ is proper.
	\end{example}

	Given a safe morphism of quotient stacks $f: X: \to Y$, we have functors:
	\[
	f_\ast : \bD(X) \rightarrow \bD(Y),
	\]
	and
	\[
	f^\circ: \bD(Y) \to \bD(X),
	\]
	induced by the corresponding morphisms for $\check C(f): \check C(U/K) \to \check C(V/L)$ (to make sense of this, one must use the base-change theorem; see Proposition \ref{propositiondmodulefacts} below). We define the functor $f^!$ as $f^\circ [\dim(X) - \dim(Y)]$.
	
	Following the convention of \cite{ben-zvi_character_2009}, we define $\D_X(\fM)$ by the usual formula for $D$-module duality when $\fM \in \bD^b_{coh}(X)$, and extend by continuity to define the functor
	\[
	\D_X: \bD(X) \to \bD(X)'.
	\]
	\begin{remark}
		If $\fM, \fN \in \bD^b_{coh}(X)$, then 
		\[
		R\bbHom(\fM, \fN) \simeq R\bbHom(\D(\fN),\D(\fM)).
		\]
		This formula fails in general, if we drop the coherence assumption.
	\end{remark}
	
	Here, we gather all the properties of $D$-module functors that we may need in this paper.
	\begin{proposition}[\cite{gaitsgory_study_2017} \cite{hotta_d_2008}]\label{propositiondmodulefacts}
		\begin{enumerate}
			\item If $f$ is proper, then $f_\ast \simeq \D_Y f_\ast \D_X$ preserves coherence and is left adjoint
			to $f^!$. We sometimes write $f_!$ instead of  $f_\ast$ in that case.
			\item If $f$ is smooth of relative dimension $d$, then $f^!$ preserves coherence and 
			$f^\ast := f^! [-2d]$ is left adjoint to $f_\ast$.
			The functor $f^\circ = f^![-d]$ $t$-exact, and $f^\circ \simeq \D_X f^\circ \D_Y$.
			\item If 
			\[
			\xymatrix{
				X \times _W V \ar[r]^-{\tilde{f}} \ar[d]_-{\tilde{g}} & V \ar[d]^-g \\
				X \ar[r]_-f & W
			}
			\]
			is a cartesian diagram of stacks, then the base change morphism is an
			isomorphism: $g^! f_\ast \cong \tilde{g}_\ast \tilde{f}^!$.
			\item We have the projection formula:
			\[
			f_\ast \left( f^! \fM \otimes \fN \right) \simeq \fM \otimes f_\ast (\fN).
			\]
			\item The category $\bD(X)$ carries a symmetric monoidal tensor product 
			\[
			{\fM}\otimes  {\fN} := \Delta^! ({\fM}\boxtimes {\fN}) \simeq \fM \otimes_{\cO_X} \fN [-\dim(X)].
			\]
			\item We have an internal Hom:
			\[
			\bbHOM ({\fM} ,{\fN}) := \D({\fM}) \otimes {\fN}.
			\]
			If $\fM$ and $\fN$ are in $\bD^b_{coh}(X)$ then 
			\[
			R\bbHom(\fM, \fN) = p_{X\ast} \bbHOM(\fM, \fN).
			\]
		\end{enumerate}
	\end{proposition}

The following elementary lemma is probably well-known, but we include a proof for completeness. 
\begin{lemma}\label{lem:stratifications}
	Let $X$ be a variety with a finite partition in to locally closed subsets indexed by a poset $(\cS,\leq)$
	\[
	X = \bigcup_{s\in \cS} X_s
	\] 
	with the property that the closure of a stratum is the union of smaller strata:\footnote{Note that the poset $\cS$ is equipped with the topology for which closed sets are upper sets, contrary to the standard convention. This is to match with the example of the Lusztig stratification of $\fg$ indexed by the poset of conjugacy classes of Levis.}
	\[
	\overline{X_s} = X_{\geq s} = \bigcup_{t\geq s} X_t
	\]
	We write $k_s:X_s \hookrightarrow X$ for the corresponding locally closed embedding.
	\begin{enumerate}
		\item If $\fM \in \bM(X)$ is a non-zero object, then there exists $s\in \cS$ such that $H^0 k_s^!(\fM)$ is non-zero.
		\item If $\fN \in \bD(X)$ is a bounded below complex, then $\fN \in \bD(X_s)^{\geq 0}$ if and only if $k_s^!(\fN) \in \bD(X_s)^{\geq 0}$ for all $s$.
	\end{enumerate}
\end{lemma}
\begin{proof}
	By way of notation, given the inclusion of a locally closed subvariety $f:V\hookrightarrow Y$ we write $R\Gamma_V$ for the (derived) functor $f_\ast f^!$ and $R^0\Gamma_V$ for $H^0 f_\ast f^!$.
	\begin{enumerate}
		\item
		Let $s\in \cS$ be maximal such that $R^0\Gamma_{X_{\geq s}} \fM$ is non-zero. Then we have a left exact sequence
		\[
		0 \to R^0\Gamma_{X_{>s}} \fM \to R^0\Gamma_{X_{\geq s}} \fM \to R^0\Gamma_{X_s} \fM
		\]
		The maximality of $s$ means that the left hand term vanishes, so the middle term must inject in to the right. Thus $R^0\Gamma_{X_s} \fM$ is non-zero or equivalently $H^0 k_s^! \fM$ is non-zero as required.
		\item One direction is clear: if $\fN \in \bD(X)^{\geq 0}$ then $k_s^! \fM \in \bD(X_s)^{\geq 0}$ for all $s \in \cS$, as the functors $k_s^!$ are left $t$-exact. For the converse, suppose $\fN$ has a non-zero negative cohomology module. We must show that $k_s^!(\fN)$ has a non-zero negative cohomology module for some $s$. Let $\fM = H^i(\fN)$ be the smallest non-zero cohomology object (by assumption $i$ is negative). By the first part of this lemma, we can find $s\in \cS$ such that $H^0k_s^!\fM \neq 0$. Now by inspection of the long exact sequence associated to the triangle
		\[
		\xymatrix{
			k_s^!\fM[-i] \ar[r] & k_s^!\fN \ar[r] & k_s^!\tau^{>i}\fN \ar[r]^-{[1]}&	
		}	
		\]
		we see that $H^i k_s^!\fN \neq 0$ as required.
	\end{enumerate}
\end{proof}

\begin{lemma}[\cite{kashiwara_d_2003}, Proposition 2.8]\label{lemmasingsuppsupp}
Given $\fM \in \bM(V)$, we have $\pi (SS(\fM)) = Supp(\fM)$, where $\pi:T^\ast V \to V$ is the projection.
\end{lemma}
	\subsection{Fourier Transform}
	Let $V$ be a complex vector space with dual space $V^\ast$. Fourier transform for $D$-modules is given by an isomorphism 
$
\fD_V \simeq \fD_{V^\ast}
$
which identifies $V^\ast \subseteq \fD_V$ (given by linear functions) with $V^\ast \subseteq \fD_{V^\ast}$ (given by constant coefficient derivations) by $-1$ times the identity map, and identifies $V\subseteq \fD_V$ and $V\subseteq \fD_{V^\ast}$  by the identity map. This isomorphism gives rise to $t$-exact inverse equivalences:
	\[
	\xymatrixcolsep{3pc}\xymatrix{
		\bD(V) \ar@/^0.3pc/[r]^{\F_V} &  \ar@/^0.3pc/[l]^{\F_{V^\ast}} \bD(V^\ast)
	}
	\]
	
\begin{lemma}
	Suppose $i:A \hookrightarrow V$ is the inclusion of a linear subspace; let $p: V^\ast \to A^\ast$ denote the adjoint to $i$. We have:
	\begin{align*}
	\F_A i^\circ  \simeq p_\ast \F_V :\bD(V) \to \bD(A^\ast) \\
	\F_V i_\ast \simeq p^\circ \F_A: \bD(A) \to \bD(V^\ast). 
	\end{align*}
\end{lemma}
\begin{proof}
We have that $i^\circ$ and $p_\ast$ are both given by the formula $\Sym(A^*) \otimes_{\Sym(V^\ast)}^{L} (-)$; it is readily checked that they agree as functors after identifying the source and target via Fourier transform as required (similarly for $i_\ast$ and $p^\circ$).
\end{proof}

\begin{lemma}\label{lemmasingsuppFourier}
Suppose $\fM \in \bM(V)$, $Supp(\fM)$ is contained in a closed conical subset $C\subseteq V$, and $Supp(\F(\fM))$ is contained in a closed conical subset $D \subseteq V^\ast$. Then $SS(\fM) \subseteq C\times D$.
\end{lemma}
\begin{proof}
By Lemma \ref{lemmasingsuppsupp}, $SS(\fM) \subseteq C\times V^\ast$.
Let $I_C \subseteq \Sym(V)$ denote the homogeneous (radical) ideal whose corresponding vanishing set is $C$. Pick a coherent submodule $\fM'$ of $\fM$, and a good filtration $F_0 \subseteq F_1 \subseteq \ldots $ of $\fM'$. Then as $Supp(\F(\fM))\subseteq C$, for every $m\in F_i$ there exists a positive integer $N$ such that $I_C^N m = 0$; thus $I_C^N \overline{m} = 0$ also, which means that $SS(\fM) \subseteq V\times C$ as claimed.
\end{proof}
\begin{remark}
The converse of Lemma \ref{lemmasingsuppFourier} does not hold in general. For example, take $\fM = \F \delta_{\alpha}$ to be the Fourier transform of a delta-function $D$-module at a non-zero point $\alpha \in V^\ast$. 
\end{remark}

We will also need the following result.
\begin{proposition}[\cite{brylinski_transformations_1986} Corollaire 7.25] \label{prop:brylinski}
	Let $V$ be a vector space, and suppose $\fF \in \bM(V)$ is regular holonomic and $\G_m$-monodromic (that is, the Euler vector field induced by the $\G_m$-action on $V$ acts locally finitely on $\fF$). Then $SS(\fF)$ and $SS(\F_V(\fF))$ are identified under the natural identification of $T^\ast V$ and $T^\ast V^\ast$.
\end{proposition}

\begin{remark}
	The proposition fails if $\fF$ is not assumed regular holonomic or monodromic. For example if $\fF$ is the $D$-module of delta functions $\delta_a$ supported at some non-zero $a\in V$ (which is not $\G_m$-monodromic), then its Fourier transform $\F_V \delta_a = \cE^a$ is an exponential $D$-module (which is not regular). We have  $SS(\delta_a) = \{a\} \times V^\ast \subseteq V\times V^\ast$, but $SS(\cE^a) = V^\ast \times \{0\}$.
\end{remark}

%

%
\bibliography{master}
\bibliographystyle{amsalpha}

\end{document}